\documentclass{amsart}
\usepackage{amssymb}
\usepackage{amsfonts}
\usepackage{amssymb}
\usepackage{amsmath}
\usepackage{amsthm}
\usepackage{enumerate}
\usepackage{tabularx}
\usepackage{centernot}
\usepackage{mathtools}
\usepackage{stmaryrd}
\usepackage{amsthm,amssymb}
\usepackage{etoolbox}
\usepackage{url}
\usepackage{tikz}
\usepackage{amssymb}
\usetikzlibrary{matrix}
\usepackage{tikz-cd}
\usepackage{tikz}
\usetikzlibrary{graphs}
\usepackage{caption}

\usepackage{marginnote}
\definecolor{mygray}{gray}{0.85}
\usepackage[backgroundcolor=mygray,colorinlistoftodos,prependcaption,textsize=small]{todonotes}
\usepackage{xargs}                      

\usepackage[colorlinks,citecolor=blue,urlcolor=blue, linkcolor=blue]{hyperref}

\usetikzlibrary{positioning, shapes.multipart, fit, backgrounds}

\usepackage{float}
\newcommand{\pureindep}[1][]{%
	\mathrel{
		\mathop{
			\vcenter{
				\hbox{\oalign{\noalign{\kern-.3ex}\hfil$\vert$\hfil\cr
						\noalign{\kern-.7ex}
						$\smile$\cr\noalign{\kern-.3ex}}}
			}
		}\displaylimits_{#1}
	}
}

\newcommand{\indep}[2]{%
	\mathrel{
		\mathop{
			\vcenter{
				\hbox{%
					\oalign{
						\noalign{\kern-.3ex}\hfil$\vert$\hfil\cr
						\noalign{\kern-.7ex}
						$\smile$\cr\noalign{\kern-.3ex}
					}
				}
			}
		}^{\!\!\!\!\!#2}_{\!\!\hspace{-0.1em}#1}
	}
}

\newcommand{\displayindep}[2]{%
	\mathrel{
		\mathop{
			\vcenter{
				\hbox{%
					\oalign{
						\noalign{\kern-.3ex}\hfil$\vert$\hfil\cr
						\noalign{\kern-.7ex}
						$\smile$\cr\noalign{\kern-.3ex}
					}
				}
			}
		}^{\!\!\hspace{-0.1em}#2}_{\!\!\hspace{-0.1em}#1}
	}
}

\newcommand{\displayfindep}[2]{%
	\mathrel{
		\mathop{
			\vcenter{
				\hbox{%
					\oalign{
						\noalign{\kern-.3ex}\hfil$\vert$\hfil\cr
						\noalign{\kern-.7ex}
						$\smile$\cr\noalign{\kern-.3ex} 
					}
				}
			}
		}^{\!\hspace{-0.14em}#2}_{\!\!\hspace{-0.05em}#1}
	}
}

\newcommand{\mrm}[1]{\mathrm{#1}}

\renewcommand{\leq}{\leqslant}
\renewcommand{\geq}{\geqslant}

\def\Ind{\setbox0=\hbox{$x$}\kern\wd0\hbox to 0pt{\hss$\mid$\hss}
	\lower.9\ht0\hbox to 0pt{\hss$\smile$\hss}\kern\wd0}
\def\Notind{\setbox0=\hbox{$x$}\kern\wd0\hbox to 0pt{\mathchardef
		\nn=12854\hss$\nn$\kern1.4\wd0\hss}\hbox to
	0pt{\hss$\mid$\hss}\lower.9\ht0 \hbox to 0pt{\hss$\smile$\hss}\kern\wd0}

\makeatletter
\def\subsection{\@startsection{subsection}{3}%
  \z@{.5\linespacing\@plus.7\linespacing}{.3\linespacing}%
  {\bfseries\centering}}
\makeatother

\makeatletter
\def\subsubsection{\@startsection{subsubsection}{3}%
  \z@{.5\linespacing\@plus.7\linespacing}{.3\linespacing}%
  {\centering}}
\makeatother

\makeatletter
\def\myfnt{\ifx\protect\@typeset@protect\expandafter\footnote\else\expandafter\@gobble\fi}
\makeatother

\renewcommand{\restriction}{ {\upharpoonright} }
\newtheorem{theorem}{Theorem}[section]
\newtheorem*{theorem*}{Main Theorem}
\theoremstyle{definition}

\newtheorem{definition}[theorem]{Definition}
\newtheorem{lemma}[theorem]{Lemma}
\newtheorem{proposition}[theorem]{Proposition}
\newtheorem{example}[theorem]{Example}

\newtheorem{fact}[theorem]{Fact}

\newtheorem{remark}[theorem]{Remark}
\newtheorem{notation}[theorem]{Notation}

\newtheorem{construction}[theorem]{Construction}

\usepackage[nocompress]{cite}

\usepackage{xcolor}

\newcounter{claimcounter}
\numberwithin{claimcounter}{theorem}
\newenvironment{claim}{\stepcounter{claimcounter}{\noindent {\underline{\em Claim \theclaimcounter}.}}}{}
 
\begin{document}
	%%%%%%%%%%%%%%%%%%

\begin{abstract} We prove that, for every Coxeter diagram $D$ with no rank $3$ residues of spherical type and such that $D$ has not only edges labelled by~$2$, the space of countable (Tits) buildings of type $D$ is Borel complete, that is, classifying countable buildings of type $D$ up to isomorphism is as hard as classifying countable graphs up to isomorphism. In particular, for every  $n\geq 3$, the space of countable generalised $n$-gons is Borel complete. This generalises the proof of Borel completeness of countable projective \mbox{planes from~\cite{pao}.}
\end{abstract}

\title[Borel completeness of Tits buildings]{Borel completeness of Tits buildings with no rank 3 residues of spherical type}

\thanks{The authors were supported by project PRIN 2022 ``Models, sets and classifications", prot. 2022TECZJA. The first author was also supported by INdAM Project 2024 (Consolidator grant) ``Groups, Crystals and Classifications''. The second author was also supported by an INdAM post-doc grant.}

\author{Gianluca Paolini}
\author{Davide Emilio Quadrellaro}
\address{Department of Mathematics “Giuseppe Peano”,
	University of Torino,
	Via Carlo Alberto 10,
	10123 Torino, Italy.}
\email{gianluca.paolini@unito.it}
\email{davideemilio.quadrellaro@unito.it}

\address{Istituto Nazionale di Alta Matematica ``Francesco Severi'',
	Piazzale Aldo Moro 5
	00185 Roma, Italy.}
\email{quadrellaro@altamatematica.it}


\date{\today}
\maketitle

\setcounter{tocdepth}{1}

\section{Introduction}
(Tits) buildings are combinatorial objects invented by J. Tits in the 1950s. Buildings were initially introduced by Tits as a means to understand the structure of semisimple algebraic groups over arbitrary fields, but their study has led to a rich and influential theory, as witnessed by the various books written on this topic (see e.g.~\cite{abra,weiss,tits_book}). As stated in the introduction to \cite[Chapter~9]{abra}, one of Tits's greatest achievements was the classification of thick, irreducible, spherical buildings of rank at least $3$, proved in \cite{tits_book}. Roughly speaking, Tits's result is that such buildings correspond to simple algebraic groups (of relative rank at least $3$) defined over an arbitrary field. On the other hand, as observed by Tits in \cite[\S1.6]{tits_local}, such classification results seem to be impossible in the general case. This claim was substantiated by Ronan in~\cite{ronan}, where, under his own words, he provided a ``free construction'' of certain buildings (those whose rank $3$ residues are all of non-spherical type), which allows one to impose any desired generalised $n$-gons as rank $2$ residues, provided their parameters fit together. Ronan uses his construction to argue that, for example, affine buildings of rank $3$ cannot be classified (cf. \cite[p.~243]{ronan}). The aim of this paper is to provide one possible interpretation of some of these anti-classification results in the framework of {\em invariant descriptive set theory} and {\em Borel complexity}. 

In recent years, this part of set theory has imposed itself as a fundamental tool to compare classification problems and prove anti-classification results in the most disparate areas of mathematics. This is done by coding the objects under classification as points of a Polish space, so that the usual notions of complexity arising from analysis can be applied (e.g., notions such as Borel, analytic, co-analytic, etc.). In the context of invariant descriptive set theory of countable structures (which is our focus here), this makes it possible to compare the complexity of the isomorphism relations for different classes of countable structures. An example of this kind of results is the recent proof of the Borel completeness of torsion-free abelian groups in \cite{PS}, due to the first named author of this paper and S. Shelah. In this context, ``Borel complete'' just means ``as complex as graph isomorphism'', since the relation of isomorphism on countable graphs is provably the most complex (in terms of Borel reducibility) isomorphism relation between countable structures. 

\medskip

Our main theorem is the following anti-classification result on buildings.

\begin{theorem*} For every Coxeter diagram $D$ with no rank $3$ residues of spherical type, and such that $D$ has at least one edge labelled by $n \in [3,\infty]$, the space of countable buildings of type $D$ is Borel complete. In particular, for every $3 \leq n < \omega$, the space of generalised $n$-gons with domain $\omega$ is Borel complete.
\end{theorem*}

We conclude by observing that our proof has two components. Firstly, for every $n \geq 3$, we exhibit a Borel construction which codes countable trees into countable generalised $n$-gons. Then, we make an appropriate use of Ronan's construction toward a proof of our main theorem. This paper is a vast generalization of the Borel completeness of countable projective planes due to the first named author \cite{pao}.
 
\section{Preliminaries}

We define in this section some preliminary notions about generalised polygons and buildings. We refer the reader to \cite{maldeghem} for an introduction to generalised polygons, and to \cite{abra,weiss} for an introduction to the theory of buildings.

\subsection{Generalised polygons}\label{subsec:polyogons}

We first recall some standard definitions from graph theory, which we will often employ in this work. By a graph we mean a structure of the form $(G, E)$ with $E$ a binary, irreflexive and symmetric relation on $G$.

\begin{definition}\label{def:graphs}
	Let $(G,E)$ be a graph, then we define the following:
	\begin{enumerate}[(i)]
		\item the \emph{valency} of an element $a\in G$ is the size of the set $\{b\in G : aEb \}$ of neighbours of $a$;
		\item  the \emph{distance} $d(a,b) = d(a, b/G)$ between two elements $a,b\in G$ is the length $n$ of the shortest path $a=c_0Ec_1E\dots E c_{n-1}Ec_n=b$ in $G$, and it is $\infty$ if there is no such path;
		\item the \emph{girth} of a graph is the length of its shortest cycle;
		\item the \emph{diameter} of a graph is the maximal distance between any two elements;
		\item the graph $G$ is \emph{bipartite} if there are $A,B\subseteq G$ such that $A\cap B=\emptyset$, $A\cup B= G$, and every edge of $G$ has one vertex in $A$ and one vertex in $B$;
		\item a graph $G$ is a \emph{tree} if it is connected, bipartite and acyclic.
	\end{enumerate}
\end{definition}

\begin{definition}[Generalised $n$-gon]\label{def_ngon}
	Let $L$ be the two-sorted language where the symbols $p_0,p_1,\dots$ denote \emph{points} and the symbols $\ell_0,\ell_1,\dots$ denote \emph{lines}, the symbol $p\, \vert\,\ell$ (or $\ell\, \vert\,p$) means that the point $p$ is \emph{incident} with the line $\ell$.  For every $n\geq 3$, we define a \emph{generalised $n$-gon} as a bipartite graph (with points and lines) with girth $2n$ and diameter $n$. A \emph{partial $n$-gon} is a model of the universal fragment of the theory of generalised $n$-gons, i.e., it is a bipartite graph with girth at least $2n$.
\end{definition}

\noindent In the previous definition, and in most of what follows, we assume that $n \geq 3$; the cases $n=2,\infty$ are often viewed as degenerate and they correspond to complete bipartite graphs (for $n =3$), and to trees (for $n=\infty$) (cf.~\cite[\S1.3]{maldeghem}).

We recall the free completion process for generalised $n$-gons, which was originally introduced by Tits in \cite{tits}. We give a slight generalisation of this construction which applies both to connected and to non-connected partial $n$-gons (and which we introduced already in \cite{PQ}). For connected partial $n$-gons the following definition is equivalent to the definition given by Tits.

\begin{definition}\label{def_free_completion}
	Let $A$ be a partial $n$-gon. For $i=0$ we let $A_0\coloneqq A$. For $1\leq i<\omega$ we obtain $A_{i+1}$ by adjoining $n-2$ new elements $z_1,\dots,z_{n-2}$ such that $z_i\,\vert\, z_{i+1}$ for all $1\leq i\leq n-3$, $a\,\vert\, z_1$, $b\,\vert\, z_{n-2}$ (and each $z_i$ for $1\leq i\leq n-2$ has no additional incidence), for any two elements $a,b\in A_i$ such that:
	\begin{enumerate}[(i)]
		\item either $d(a,b/A_i)=n+1$;
		\item or $d(a,b/A_i)=\infty$ and $a,b$ are of the appropriate sorts (i.e., $a,b$ have the same sort if $n$ is odd and have different sorts if $n$ is even).
	\end{enumerate}
	The structure $F(A):=\bigcup_{i < \omega}A_i$ is called the \emph{free completion} of $A$ and we say that $F(A)$ is freely generated over $A$. We say that $A$ is \emph{non-degenerate} if $F(A)$ is infinite, and \emph{degenerate} otherwise.
\end{definition}

\begin{remark}\label{non-degeneracy}
	It was originally proved by Tits in \cite{tits} that, if $A$ is a partial non-degenerate $n$-gon, then $F(A)$ is an infinite generalised $n$-gon (cf. also \cite[p.~755]{funk2} and see \cite{PQ} for a reference considering exactly the definition above). We recall also that, if $A$ is a connected partial $n$-gon, then $A$ is non-degenerate exactly if it contains a cycle of length $2n+2$, or if it has diameter $\geq n+2$ (cf. \cite[p.~755]{funk2}). We shall often use this characterisation in Section~\ref{sec_n_gons} and Section~\ref{sec_building}.
\end{remark}

\begin{notation}
	Given a partial $n$-gon $A$, a chain of $(n-2)$-elements $z_1,\dots,z_{n-2}$ such that $a\,\vert\, z_1$, $b\,\vert\, z_{n-2}$, and every elements in $\{z_1, ..., z_{n-2}\}$ has exactly valency 2, is called a \emph{clean arc} between $a$ and $b$. We stress in particular that $a$ and $b$ do \emph{not} belong to the clean arc, and that we often refer to them as the \emph{endpoints} of the clean arc $\bar{z}=(z_1,\dots, z_{n-2})$. Notice also that the parity of a clean arc depends on the underlying value of $n$: clean arcs are even for even $n$ and odd for odd $n$.
\end{notation}

The following definition is as in \cite{tent,PQ} and it slightly differs from others from the literature (e.g. \cite{funk2}). We refer to \cite{PQ} for motivations and more background.

\begin{definition}\label{ngons:def_hyperfree}
	Let $A$ be a finite partial $n$-gon, we say that a tuple $\bar{a}\in A^{<\omega}$ is \emph{hyperfree in $A$} if it satisfies one of the following conditions:
	\begin{enumerate}[(i)]
		\item $\bar{a}=(a)$ is a \emph{loose end}, i.e., it is a single element incident with at most one element in $A$;
		\item $\bar{a}=(a_1,\dots, a_{n-2})$ is a \emph{clean arc}, i.e., it forms a chain $a_1\, \vert \, a_2\, \vert\ \dots\, \vert\ a_{n-3}\, \vert\,a_{n-2}$ where each $a_i$ for $1\leq i\leq n-2$ is incident to exactly two elements from $A$.
	\end{enumerate}
\end{definition}

\noindent The following notion of ``confined configuration'' plays an important role in the study of free constructions in incidence geometry, and it was studied at length by Funk and Strambach in \cite[Def.~4]{funk2}. We stress that, as our definition of hyperfree configuration from Definition~\ref{ngons:def_hyperfree} is slightly different from \cite[p.~752]{funk2}, also our characterisation of confined partial $n$-gon differs from the one provided in \cite[p.~753]{funk2}. This  definition of confined configuration generalises the concept of \emph{confined plane} from the theory of projective planes (cf. \cite[p.~220]{projective_planes}).

\begin{definition}[Confined generalised $n$-gon]\label{def_confined_completion}
	A finite partial $n$-gon $A$ is said to be \emph{confined} if it does not contain any tuple hyperfree in $A$. An arbitrary partial $n$-gon $A$ is said to be \emph{confined} if every element $x\in A$ is contained in a finite subset $B\subseteq A$ which is confined.
\end{definition}

\begin{remark}\label{remark:confined}
	We recall that the case $n=3$ of generalised $n$-gon is exactly the case of projective planes. In this case, the free completion process from \ref{def_free_completion} coincides with the free completion process for partial projective planes originally introduced by Hall in \cite{hall_proj}. Then Definition~\ref{ngons:def_hyperfree} simply says that an element in a projective plane is hyperfree if and only if it has valency $\leq 2$, and Definition~\ref{def_confined_completion} says that a  partial projective plane is confined if and only if every element has valency $\geq 3$ (cf. \cite[p.~752]{funk2} and \cite[p.~220]{projective_planes}). More generally, it is easy to deduce from \ref{def_confined_completion} that a partial generalised $n$-gon $A$ is confined if and only if every element in $A$ has valency $\geq 2$ and, moreover, every chain of $(n-2)$-many elements in $A$ contains one element with valency $\geq 3$.
\end{remark}

 We conclude this section with the following lemma, which will be a crucial ingredient of our Borel construction.

\begin{lemma}\label{lemma_confined} Let $A$ be a partial confined generalised $n$-gon. If $B \subseteq F(A)$ (where $F(A)$ is as in \ref{def_free_completion}) is finite  and confined, then $B \subseteq A$.
\end{lemma}
\begin{proof}
	
	For each $i<\omega$ we let $A_i$ be as in Definition~\ref{def_free_completion}. Since $B$ is finite there is a smallest $n<\omega$ such that $B\subseteq A_{n}$. If $n=0$ then $B\subseteq A_0=A$ and the claim follows. Thus we can assume that $n\geq 1$. Then it follows from  Definition~\ref{def_free_completion} that there is a clean arc $\bar{c}=(c_1,\dots, c_{n-2})$ in $A_n$ with endpoints $a,b\in A_{n-1}$ such that $\{c_1,\dots, c_{n-2}\}\cap B\neq \emptyset$. We now have two cases. 
\noindent
\newline \underline{Case 1}. $\{ c_1,\dots, c_{n-2}, a,b \} \subseteq B$.
\newline Since $B\subseteq A_{n}$ in this case we have that $c_1,\dots, c_{n-2}$ is a clean arc also in $B$, and thus, by Definition~\ref{ngons:def_hyperfree}(ii) it is hypefreee in $B$. By Definition~\ref{def_confined_completion} we immediately conclude that $B$ is not confined.
\newline \underline{Case 2}. $\{ c_1,\dots, c_{n-2}, a,b \} \not \subseteq B$.
\newline Since $B\subseteq A_{n}$ and every element from $c_1,\dots, c_{n-2}$ has valency 2 in $A_n$, it follows that in $B$ one of these elements has valency $\leq 1$, which by Definition~\ref{ngons:def_hyperfree}(i) means that it is hyperfree. Then, it follows again from Definition~\ref{def_confined_completion} that $B$ is not confined, which completes the proof of the lemma.
\end{proof}

\subsection{Chambers complexes and buildings}

As outlined in the introduction to \cite{abra}, buildings can be approached from (at least) three different perspective: the \emph{simplicial} approach, the \emph{combinatorial} approach, and the \emph{metric} approach. We follow here the (original) simplicial approach to buildings and recall some basic definitions from \cite[\S A.1, \S4.1]{abra} and \cite[\S1, \S3]{tits_book}.

\begin{definition}[Simplicial complexes]\label{def_simplices}
	A \emph{simplicial complex} is a set  $\Delta$ endowed with a partial order $\leq$ satisfying the following two conditions:
	\begin{enumerate}[(i)]
		\item any two elements $a,b\in \Delta$ have a greatest lower bound $a\land b$
		\item for any $a\in \Delta$, the poset $\Delta_{\leq a}\coloneqq\{b\in \Delta : b\leq a  \}$ is isomorphic to the poset $(\wp(r), \subseteq  )$ of subsets of $\{1, . . . , r\}$ ordered by inclusion, for some integer $r\geq 0$.
	\end{enumerate}
The elements of $\Delta$ are called \emph{simplices}. A simplex $a\in \Delta$ is \emph{maximal} if $b\in \Delta$ and $a\leq b$ entail $a=b$. For every simplex $a\in \Delta$, the elements in $\Delta_{\leq a}$ are called the \emph{faces} of $a$; the number $r$ from \ref{def_simplices}(ii) above is the \emph{rank} of $a$ and $r-1$ is its \emph{dimension}, which we denote by $\mrm{dim(a)}$. The \emph{codimension} of a face $b$ of $a$ (with respect to $a$) is the number $\mrm{dim}(a)-\mrm{dim}(b)$. The elements of $\Delta$ of rank 1 (or, equivalently, of dimension $0$) are called the \emph{vertices} of $\Delta$. A \emph{simplicial complex} $\Delta$ is said to be \emph{finite dimensional} if there is a largest $n<\omega$ such that some $a\in \Delta$ has dimension $n$, and $n$ is said to be the dimension of $\Delta$.
\end{definition}

\begin{definition}[Chamber complex]\label{def:chamber_complex}
	Let $\Delta$ be a finite dimensional simplicial complex. We say that $\Delta$ is a \emph{chamber complex} if the following conditions hold:
	\begin{enumerate}[(i)]
		\item all maximal simplices have the same dimension;
		\item given two distinct maximal simplices $a,b\in \Delta$ there is some $n<\omega$ and some sequence of maximal simplices $c_0,\dots,c_{n}$ with $c_0=a$, $c_{n}=b$ and such that for all $i< n$ the simplex $c_i\land c_{i+1}$ is of dimension $\mrm{dim}(a)-1$ ($=\mrm{dim}(b)-1$).
	\end{enumerate}
We say that $\Sigma\subseteq \Delta$ is a \emph{chamber subcomplex} of $\Delta$ if it is a chamber complex under the same order relation $\leq$ from $\Delta$ and, additionally, its maximal simplices have the same dimension as those from $\Delta$.
\end{definition}

\begin{notation}\label{notation_chamber complexes}
	In the context of a chamber complex $\Delta$, its maximal simplices are usually called \emph{chambers}. Since the chambers of $\Delta$ all have the same dimension, any simplex $b$ has the same codimension with respect to any two chambers, and we thus simply refer to this value as the \emph{codimension of a simplex $b\in \Delta$}. The simplices of codimension $1$ in $\Delta$ are called \emph{panels}. A chamber complex is \emph{thin} if each panel is a face of exactly two chambers, and it is \emph{thick} if every panel is a face of at least three chambers. Two chambers $a,b\in \Delta$ are \emph{adjacent} if they are distinct and they have a common face of codimension 1, i.e., a common panel. A sequence $\gamma=(d_0,\dots,d_{n})$ of chambers such that $d_i$ is adjacent to $d_{i+1}$ for all $i< n$ is called a \emph{gallery}. Thus, condition~(ii) from Definition~\ref{def:chamber_complex} is simply saying that any two chambers in $\Delta$ are joined by a gallery of length $n$, for some $n \in \mathbb{N}$. The \emph{distance} $d(x,y/\Delta)$ of two chambers  $x,y\in \Delta$ is the least $n<\omega$ such that $\gamma=(d_0,\dots,d_{n})$ is a gallery with $d_0=x$ and $d_n=y$ (this coincides with the graph-theoretic distance if we view $\Delta$ as a graph with edges induced by the adjacency relation).
\end{notation}

The following definition of building is from Tits (cf.~\cite[\S3.1]{tits_book}), and it introduces buildings as chambers complexes made of \emph{apartments} (cf. also \cite[\S1.1]{busch}). We notice that sometimes buildings are not assumed to be thick (e.g., in \cite[Def.~4.1]{abra}). The only practical difference between these two definitions is that in the second case also the Coxeter complex can be predicated of the term ``building'' (cf.~\ref{coxeter_complex} below).

\begin{definition}[Building]\label{def_building}
	A building is a thick chambers complex $\Delta$ satisfying the following conditions:
	\begin{enumerate}[(i)]
		\item $\Delta$ is the union of thin chamber subcomplexes (which we call \emph{apartments});
		\item for any two simplices, there is an apartment containing both of them;
		\item every two apartments are isomorphic through an isomorphism fixing every vertex in their intersection.
	\end{enumerate}
\end{definition}

\begin{example}\label{example:ngons_as_buildings}
	Let $A$ be a generalised $n$-gon for some $n\geq 3$ and suppose also that $A$ is \emph{firm}, i.e., that every element has valency at least 2. A subset $F$ of $A$ is a \emph{flag} if every two distinct elements from $F$ are incident to each other. It is immediate to verify that there are only four types of flags, the empty flag $\emptyset$, the one-element flags $\{p\}$ consisting of only one point $p\in A$, the one-element flags $\{\ell\}$ consisting of only one line $\ell\in A$, and the two-elements flags $\{p,\ell\}$ consisting of one point $p\in A$ and one line $\ell\in A$ such that $p\, \vert \, \ell$. We call the set of all flags of $A$ ordered by the inclusion relation the \emph{flag complex} associated to $A$, and we often denote it by $\Delta_A$. We leave it to the reader to verify that this is a building whose apartments are the regular $n$-gons, i.e., the cycles of length $2n$. See \cite[\S4.2]{abra} and \cite[\S1.3.7]{maldeghem} for more details on this construction. In Remark~\ref{remark:spherical_buildings} we shall also see that, if $\Delta$ is a  building of rank 2, then it is either a generalised $n$-gon with $n\geq 3$, a complete bipartite graph, or a tree without endpoints.
\end{example}

We conclude this section by recalling the connection between buildings and Coxeter groups, and some associated basic definitions.

\begin{definition}
	A (finitely generated) \emph{Coxeter group} is a group $W$ that admits a presentation of the form \[\langle s_1,\ldots,s_n \ \vert \ (s_is_j)^{m_{ij}}=1, \ \text{for all} \ i,j \rangle\] where $0\neq m_{ij}\in \mathbb{N}\cup \lbrace \infty\rbrace$ for every $1\leq i,j\leq n$ and $m_{ij}=1$ if and only if $i = j$ (the relation $(s_is_j)^{\infty}=1$ means that $s_is_j$ has infinite order). Note that each generator $s_i$ has order two, and that $m_{ij}=2$ if and only if $s_i$ and $s_j$ commute. The set $S=\lbrace s_1,\ldots,s_n\rbrace$ is called a \emph{Coxeter generating set}, and the pair $(W,S)$ is called a \emph{Coxeter system}. The \emph{Coxeter diagram} $D=(S,M)$ associated to the Coxeter system $(W,S)$ is the following labelled graph: the generators $S=\{s_1,\dots,s_n\}$ are the vertices of $D$, and for all generators $s_i,s_j$ there is a unique edge in $M$ between $s_i$ and $s_j$ labelled by $m_{ij}$. We often draw these graphs by omitting the edges labelled by 2, and by drawing edges labelled by 3 without a label. We say that $(W,S)$ has \emph{type} $D$ when $D$ is the Coxeter diagram of the system $(W,S)$. Also, notice that we will later often identify the set $S$ with its index set $I=\{1,\dots,n\}$ and write $i,j,k,\dots$ for its elements.
\end{definition}

\begin{definition}[Coxeter complex]\label{coxeter_complex}
	Let $(W,S)$ be a Coxeter system with a generating set $S$. A \emph{standard coset} of $W$ is a coset of the form $w\langle J\rangle_W$ for some $w\in W$ and $J\subseteq S$. The \emph{Coxeter complex} $\Sigma(W,S)$ of $(W,S)$ is the poset of standard cosets of $(W,S)$ ordered by reverse inclusion. The \emph{diagram} of the Coxeter complex $\Sigma(W,S)$ is simply the diagram of the Coxeter system $(W,S)$; whence we also say that $\Sigma(W,S)$ has \emph{type} $D$. The \emph{type} of a simplex $w\langle J\rangle_W \in\Sigma(W,S)$ is $S\setminus J$, and its \emph{cotype} is the set $J$ itself.
\end{definition}

\noindent As we shall make more explicit in \ref{notation_for_buildings}, the chambers of a Coxeter complex are essentially the elements of $W$, and their adjacency relation coincides with the incidence relation in the Cayley graph associated to $(W,S)$. The fundamental role of Coxeter complexes in the context of buildings is motivated by the fact that they are exactly the apartments from Definition~\ref{def_building}.

\begin{remark}\label{remark:spherical_buildings}
	Definition~\ref{coxeter_complex} allows one to provide a sort of converse to Example~\ref{example:ngons_as_buildings}, which was originally proved by Tits in \cite{tits_book}. In fact, it can be proved that in any building $\Delta$ all its apartments are isomorphic to the Coxeter complex associated a fixed Coxeter diagram $D$. If $\Delta$ is a building, we can then define the \emph{Coxeter diagram of $\Delta$} as the Coxeter diagram of any of its apartments, and we then say that $\Delta$ is a building of \emph{type} $D$. We say that a building is \emph{irreducible} if its diagram is connected (as a graph). A building (and its corresponding Coxeter diagram) is \emph{spherical} if its corresponding Coxeter group is finite. In particular any rank 2 building has diagram 
\begin{center}
	\begin{tikzcd}
	\bullet \arrow[r, "m", no head] & \bullet
	\end{tikzcd}
\end{center}
\noindent where $2\leq m \in \mathbb{N}$ or $m=\infty$. It follows that every rank 2 building is the flag complex of either a generalised $n$-gon (the case $3\leq n\in \mathbb{N}$), a complete bipartite graph (the case $n=2$), or a tree without endpoints (the case $n=\infty$). We notice that for this reason generalised $n$-gons are sometimes defined also for the two degenerate cases with $n=2$ and $n=\infty$ (recall our remarks in Section~\ref{subsec:polyogons} and see also \cite[\S1]{maldeghem}).
\end{remark}

\begin{definition}\label{def:basic_buildings} Let $(W,S)$ be a Coxeter system of type $D$ and let $\Delta$ be a building of the same type. For any $s\in S$ we say that two chambers $a,b\in \Delta$ are \emph{$s$-adjacent} if they have a common panel of cotype $\{s\}$ (recall Notation~\ref{notation_chamber complexes} and Definition~\ref{coxeter_complex}). Given $J \subseteq S$, we say that two chambers $a,b \in \Delta$ are $J$-\emph{equivalent}, and we write $a \sim_J b$, if there is a gallery $\gamma=(c_0,\dots, c_n)$ with $c_0=a$, $c_n=b$ and such that for every $k<n$ the chambers $c_k$ and $c_{k+1}$ are $j$-equivalent for some $j\in J$. Then  $J$-equivalence is an equivalence relation whose equivalence classes are called $J$-\emph{residues}. We denote the (unique) $J$-residue containing a given chamber $c\in \Delta$ by $R_J(c)$. A subset $\Sigma\subseteq \Delta$ is called a \emph{residue} if it is a $J$-residue for some $J \subseteq S$, and it can therefore be viewed as a building with Coxeter diagram $(W_J, J)$, where $W_J=\langle J\rangle_W$ the standard parabolic subgroup generated by $J$ in $W$, which is known to be a Coxeter group with generators $J$. The restriction of the diagram $D$ to $J$ is called the \textit{type} of the residue, and the cardinality $|J|$ is called the \textit{rank} of $\Sigma$.  In particular, a building $\Delta$ of type $D=(S,M)$ \emph{has no rank 3 residues of spherical type} whenever for every $J\subseteq S$ with  $J$ of size $3$ we have that the Coxeter group $W_J$ is infinite (we shall see below a more direct characterisation of such buildings).
\end{definition}

We will deal in Section~\ref{sec_n_gons} with the Borel completeness problem for generalised $n$-gons, and in Section~\ref{sec_building} with the general case of buildings with no spherical residues of rank 3. Given Coxeter's classification of finite Coxeter groups from \cite{coxeter}, the buildings with no spherical residues of rank 3 can be characterised as those buildings $\Delta$ for which every three vertices $i,j,k$ in their corresponding Coxeter diagram satisfy:  \[\frac{1}{m_{ij}}+\frac{1}{m_{ik}}+\frac{1}{m_{jk}} \leq 1.\]
We refer the reader not acquainted with the literature on buildings to \cite{abra}, \cite{tits_book} and \cite{weiss} for additional background. We stress however that we will not assume any additional knowledge from the theory of buildings. The key ingredient in the proof of our main theorem will be Ronan's free completion method (from \cite{ronan}) for arbitrary buildings with no spherical residues of rank 3. We will recall later in Section~\ref{sec_building} the essential ingredients of Ronan's construction.

\section{Generalised polygons}\label{sec_n_gons}

We provide in this section a proof of the Borel completeness of the Borel space of countable generalised $n$-gons, for all $3\leq n<\omega$. The key strategy is to define a (Borel) coding of countable trees (which are Borel complete) into countable generalised $n$-gons. Firstly, in Section~\ref{ngon:finitary_construction}, we describe a finite combinatorial pattern of confined partial $n$-gons, which we use to capture the edge relation between two elements in a tree. Then, in Section~\ref{Sec:borel_reduction_ngon}, we use this construction to define the Borel reduction from the space of countable trees into the space of countable generalised $n$-gons, for all $3\leq n<\omega$. We notice that the Borel completeness of (non-Desarguesian) projective planes was already showed in \cite{pao} by the first named author of this paper, but we nonetheless consider here also the case $n=3$, for completeness, and to later provide a uniform treatment of buildings in Section~\ref{sec_building}.

\subsection{The finitary construction}\label{ngon:finitary_construction}

The aim of this section is to describe a combinatorial patter of partial generalised $n$-gons to be used in our main construction, the following lemma is the key.

\begin{lemma}\label{the_crucial_lemma} Let $3 \leq n < \omega$.  There is a finite confined  partial generalised $n$-gon $A$ such that, letting $A_1 \cup A_2$ be the disjoint union of two copies of $A$, then there is a finite, confined, and non-degenerate partial generalised $n$-gon $B$ containing $A_1 \cup A_2$ such that the only copies of $A$ in $B$ are $A_1$ and  $A_2$.
\end{lemma}

We prove Lemma~\ref{the_crucial_lemma} by constructing different configurations $A$ and $B$ for all $3 \leq n < \omega$. We deal with the cases $n=3$ and $n=4$ in Construction~\ref{case_n=3} and Construction~\ref{case_n=4}. For all $n\geq 5$ we provide a uniform treatment which depends solely on the parity of $n$: we consider the odd case in Construction~\ref{odd_case} and the even case in Construction~\ref{even_case}. Together, these constructions establish Lemma~\ref{the_crucial_lemma}.

\begin{construction}[$n=3$]\label{case_n=3} We let $A$ be the incidence graph associated to the Fano plane, i.e., the projective plane $P(\mathbb{F}_2)$ over the two-elements field $\mathbb{F}_2$. We consider its associated incidence graph depicted below, where white dots represent lines and black dots represent points (cf.~\cite[p.~180]{abra}) --- this is also called the \emph{Heawood graph} in the literature (cf.~\cite[p.~209]{brouwer}). Formally, this is defined as follows. Let $\{a_{2i}:1\leq i\leq 7\}$ be a set of points and let $\{a_{2i+1}:0\leq i\leq 6\}$ be a set of lines. We let $a_i\;\vert\; a_{i+1}$ for all $1\leq i\leq 13$ and also $a_{14}\; \vert\; a_1$. Moreover, we add the incidences  $a_{1}\; \vert\; a_{10}$, $a_{2}\; \vert\; a_{7}$, $a_{3}\; \vert\; a_{12}$, $a_{4}\; \vert\; a_{9}$, $a_{5}\; \vert\; a_{14}$, $a_{6}\; \vert\; a_{11}$, and $a_{8}\; \vert\; a_{13}$.

	 	\begin{figure}[H]
\centering 
	 	\begin{tikzpicture}[scale=2]
	 	
	 	\foreach \i in {1,3,5,7,9,11,13} {
	 		\node[circle,draw,fill=white,inner sep=1.5pt,label={\i*360/14:$a_{\i}$}]  (\i) at ({360/14 * (\i-1)}:1) {};
	 	}
 	
 	\foreach \i in {2,4,6,8,10,12,14} {
 		\node[circle,draw,fill=black,inner sep=1.5pt,label={\i*360/14:$a_{\i}$}]  (\i) at ({360/14 * (\i-1)}:1) {};
 	}
	 	
	 	\draw
	 	(1) -- (2)  
	 	(2) -- (3)  
	 	(3) -- (4) 
	 	(4) -- (5)
	 	(5) -- (6) 
	 	(6) -- (7) 
	 	(7) -- (8) 
	 	(8) -- (9) 
	 	(9) -- (10) 
	 	(10) -- (11)
	 	(11) -- (12)  
	 	(12) -- (13)  
	 	(13) -- (14) 
	 	(14) -- (1)

	 	(1) -- (10) 
	 	(2) -- (7) 
	 	(3) -- (12) 
	 	(4) -- (9)
	 	(5) -- (14) 
	 	(6) -- (11) 
	 	(8) -- (13) 
	 	;
	 	
	 	\end{tikzpicture}
	 	\caption{\label{fig:figure1} The Heawood Graph.}	 
	 	\end{figure}
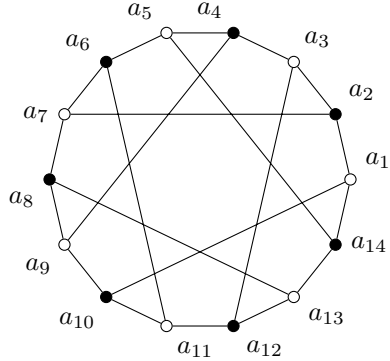
 	
	 \noindent It is immediate to verify that the diagram above has diameter 3 and girth 6, and thus forms a generalised $3$-gon, i.e., a projective plane. Since every element has valency 3, it also follows that it is confined (recall Remark~\ref{remark:confined}). To show that Lemma~\ref{the_crucial_lemma} holds for $n=3$ it thus only remains to define a confined, non-degenerate configuration $B$ which contains two disjoint copies $A_1, A_2$ of $A$ and such that $B$ does not contain any other isomorphic copy of $A$.
	 
	 \smallskip \noindent Let $A_1,A_2$ be two isomorphic copies of the Heawood graph  with vertices $\{a_i \}_{1\leq i\leq 14}$ and $\{b_i \}_{1\leq i\leq 14}$. We define $B\coloneqq A_1\cup A_2\cup \{p,\ell\}$, where $p$ is a point incident to $a_1$ and $b_1$; $\ell$ is a line incident to $a_4$ and $b_4$, and additionally we let  $p\,\vert\, \ell$. It is immediate to check that $B$ is a partial projective plane and, since every element has valency 3, that it is confined (again by Remark~\ref{remark:confined}). Also, since $a_2$ and $b_3$ have distance 5 in $B$, it follows from Remark~\ref{non-degeneracy} that $B$ is non-degenerate. We verify that $B$ does not contain any other copy of $A$ than $A_1,A_2$. Suppose $\pi:A\to B$ is an embedding, if $p\in \pi(A)$ then $p$ must have valency 3 in $\pi(A)$ and thus $\ell,a_1,b_1\in \pi(A)$. Since $\ell$ must also have valency 3 in $\pi(A)$, we also derive that $a_4,b_4\in \pi(A)$. Now, $a_1$ must have valency $3$ in $\pi(A)$, so two elements from $\{a_2,a_{14},a_{10}\}$ must be in $\pi(A)$. Suppose without loss of generality that $a_{10}\in \pi(A)$ (the argument is analogous for $a_2\in \pi(A)$ or $a_{14}\in \pi(A)$), then it follows by the fact that $a_{10}$ must have valency $3$ in $\pi(A)$ that also $a_9$ and $a_{11}$ belong to $\pi(A)$. Repeating the same argument it follows that every element $a_i$ for $5\leq i\leq 14$ is in $\pi(A)$. It follows that 
	 $ \{a_i : 4\leq i\leq 14\}\cup \{a_1\}\cup \{p,\ell\}\cup \{b_1,b_4 \}\subseteq \pi(A), $
	 which is clearly impossible since $\pi $ is an injective and so $|\pi(A)|=14$. It follows that $\pi(A)$ does not contain $p$ and, by a similar argument, that it does not contain $\ell$. Since every path in $B$ from one element in $A_1$ to an element in $A_2$ must contain either $p$ or $\ell$, and since every two elements in $A$ are joined by a path, this entails that $\pi(A)=A_1$ or $\pi(A)=A_2$. We conclude that $A$ and $B$ satisfy Lemma~\ref{the_crucial_lemma}.
	 \end{construction}

\begin{construction}[$n=4$]\label{case_n=4}
In the case with $n=4$ we consider the so-called Cremona-Richmond configuration, i.e., the unique generalised quadrangle $\mrm{GQ}(2,2)$ where each point lies on three lines and each line lies on three point. As in the previous case of the Fano plane, we view it from the graph-theoretic perspective via its associated incidence graph, the so-called \emph{Tutte-Coxeter graph}, which is represented right below (cf.~\cite[p.~209]{brouwer}). Formally, we define it as follows. Let $\{a_{2i}:1\leq i\leq 15\}$ be a set of points and $\{a_{2i+1}:0\leq i\leq 14\}$  a set of lines. We let $a_i\;\vert\; a_{i+1}$ for all $1\leq i\leq 29$ and also $a_{30}\; \vert\; a_1$. Moreover, we add the incidences $a_{1}\; \vert\; a_{18}$, $a_{2}\; \vert\; a_{23}$, $a_{3}\; \vert\; a_{10}$, $a_{4}\; \vert\; a_{27}$, $a_{5}\; \vert\; a_{14}$, $a_{6}\; \vert\; a_{19}$, $a_{7}\; \vert\; a_{24}$, $a_{8}\; \vert\; a_{29}$, $a_{9}\; \vert\; a_{16}$, $a_{11}\; \vert\; a_{20}$, $a_{12}\; \vert\; a_{25}$, $a_{13}\; \vert\; a_{30}$, $a_{15}\; \vert\; a_{22}$, $a_{17}\; \vert\; a_{26}$, $a_{21}\; \vert\; a_{28}$.

\begin{figure}[h]
	\centering
	\begin{tikzpicture}[scale=2.7]
	
	\foreach \i in {1,3,5,7,9,11,13,15,17,19,21,23,25,27,29} {
		\node[circle,draw,fill=white,inner sep=1.5pt,label={\i*360/30:$a_{\i}$}]  (\i) at ({360/30 * (\i-1)}:1) {};
	}

\foreach \i in {2,4,6,8,10,12,14,16,18,20,22,24,26,28,30} {
	\node[circle,draw,fill=black,inner sep=1.5pt,label={\i*360/30:$a_{\i}$}]  (\i) at ({360/30 * (\i-1)}:1) {};
}

	\draw
	(1) -- (2)  
	(2) -- (3)  
	(3) -- (4) 
	(4) -- (5)
	(5) -- (6) 
	(6) -- (7) 
	(7) -- (8) 
	(8) -- (9) 
	(9) -- (10) 
	(10) -- (11)
	(11) -- (12)  
	(12) -- (13)  
	(13) -- (14) 
	(14) -- (15)
	(15) -- (16) 
	(16) -- (17) 
	(17) -- (18) 
	(18) -- (19) 
	(19) -- (20) 
	(20) -- (21)
	(21) -- (22)  
	(22) -- (23)  
	(23) -- (24) 
	(24) -- (25)
	(25) -- (26) 
	(26) -- (27) 
	(27) -- (28) 
	(28) -- (29) 
	(29) -- (30) 
	(30) -- (1) ;

	\draw
	(1) -- (18) 
	(2) -- (23) 
	(3) -- (10)
	(4) -- (27)
	(5) -- (14)
	(6) -- (19)
	(7) -- (24)
	(8) -- (29)
	(9) -- (16)
	(10) -- (3)
	(11) -- (20)
	(12) -- (25)
	(13) -- (30)
	(15) -- (22)
	(17)--(26)
	(21)--(28)
	;
	
	\end{tikzpicture}
\caption{\label{fig:figure2} The Tutte-Coxeter Graph.}	 
\end{figure}
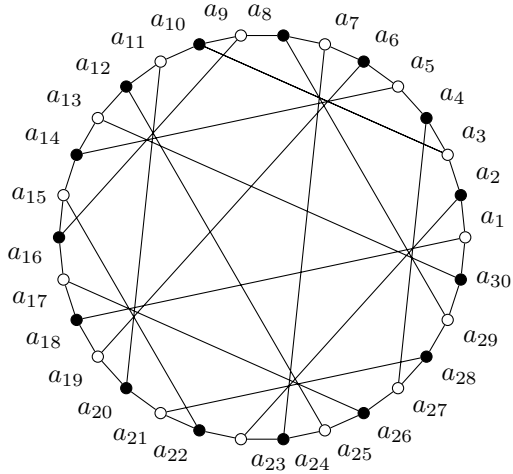

\noindent It is easy to verify that the diagram above has diameter 4 and girth 8, and thus it is a generalised quadrangle. Moreover, since every element has valency 3, it immediately follows by Remark~\ref{remark:confined} that it is also confined. It remains to define the configuration $B$ and to show that it satisfies Lemma~\ref{the_crucial_lemma}. 
	 
\smallskip \noindent Let $A_1$ and $A_2$ be two isomorphic copies of the Tutte-Coxeter graph with vertices $\{a_i \}_{1\leq i\leq 30}$ and $\{b_i \}_{1\leq i\leq 30}$, respectively. We define $B\coloneqq A_1\cup A_2\cup \{c\}$ where $c$ is a point incident to the two lines $a_1,b_1$. Since $c$ does not introduce any new cycle, $B$ is still a partial quadrangle, and since $c$ has valency 2 and it does not belong to any clean arc $d_1\,\vert\, d_2$ of length 2, $B$ is still confined (recall Remark~\ref{remark:confined}). Also, since $a_4$ and $b_4$ have distance 8 in $B$, it also follows from Remark~\ref{non-degeneracy} that $B$ is non-degenerate. Finally, we notice that if $\pi:A\to B$ is an embedding then $c\notin \pi(a)$, since $c$ has only valency 2 in $B$, while every element in $A$ has valency 3. Therefore, if $\pi(A)$ contains two elements $x\in A_1$ and $y\in A_2$, there can be no path between these two elements in $\pi(A)$, since in $B$ the only path between them passes through $c$. It follows that $\pi(A)=A_1$ or $\pi(A)=A_2$ and thus that $B$ satisfies Lemma~\ref{the_crucial_lemma}.
\end{construction}

\begin{construction}[$n=2m+1$ for $m\geq 2$]\label{odd_case} We address the odd case for all $n\geq 5$. Fix $n=2m+1$ for some $m\geq 2$, we define a partial confined $n$-gon $A$ as follows. For $k=1,3,5,7 $ we let $a_k^1\,\vert\,\dots \,\vert\, a_k^{n-2}$ be a chain such that $a_k^1$ and $a_k^{n-2}$ are both points (this is possible because $n$ is odd). Similarly, for $k= 2,4,6$ we let $a_k^1\,\vert\,\dots\,\vert\, a_k^{n-2}$ be a chain such that $a_k^1$ and $a_k^{n-2}$ are both lines. Additionally, we let $a_k^{n-2}\,\vert\, a_{k+1}^{1}$ for all $1\leq k\leq 6$. For $k=1,3$ we let $b_k^1\,\vert\,\dots\,\vert\,b_k^{n-3}$ be a chain with $b_k^1$ a line and $b_k^{n-3}$ a point; for $k=2,4$ we let $b_k^1\,\vert\,\dots\,\vert\,b_k^{n-3}$ be a chain with $b_k^1$ a point and $b_k^{n-3}$ a line. Additionally, we let $a_1^1\,\vert\, b_1^1$, $a_2^1\,\vert\, b_2^1$, $a_3^1\,\vert\, b_3^1$, $a_4^1\,\vert\, b_4^1$ and also $a_4^{n-2}\,\vert\, b_1^{n-3}$, $a_5^{n-2}\,\vert\, b_2^{n-3}$, $a_6^{n-2}\,\vert\, b_3^{n-3}$, $a_7^{n-2}\,\vert\, b_4^{n-3}$. Finally, we let $c_1\,\vert\,\dots \,\vert\, c_{n-4}$ be a chain with $c_1$ and $c_{n-4}$ both lines and such that $a_1^1\,\vert\, c_1$, $a_7^{n-2}\,\vert\, c_{n-4}$ (so for $n=5$ we have in particular that $a_1^1\,\vert\, c_1=c_{n-4}\,\vert\, a_7^{n-2}$). We draw this configuration below.
	
	\begin{figure}[h]
		\centering
		\begin{tikzpicture}[scale=0.55]
		
		\node[circle, draw, fill=white, inner sep=1.5pt] (roota) at (9,10) {};
		\node (rootc) at (10,10) {$\dots$};
		\node[circle, draw, fill=white, inner sep=1.5pt] (rootb) at (11,10) {};
		\draw (roota) -- (rootc);
		\draw (rootb) -- (rootc);

		\node[draw=blue, rounded corners, fit=(roota)(rootb), inner sep=3pt, label=above:{$c_1,\dots,c_{n-4}$}] (rectRoot) {};

		\foreach \i in {0,2,6,8,12,14,18,20} {
			\node[circle, draw, fill=black, inner sep=1.5pt] (m\i) at (\i*1,5) {};
		}
	\foreach \i in {3,5,9,11,15,17} {
		\node[circle, draw, fill=white, inner sep=1.5pt] (m\i) at (\i*1,5) {};
	}
		\foreach \i in {1,4,7,10,13,16,19} {
			\node (m\i) at (\i*1,5) {$ \dots $};
		}

		\foreach \i in {0,...,19} {
			\draw (m\i) -- (m\the\numexpr\i+1\relax);
		}

		\foreach \i in {0,2,4,6} {
			\pgfmathtruncatemacro{\start}{3*\i}
			\pgfmathtruncatemacro{\end}{3*\i+2}
			\node[draw=blue, rounded corners, fit=(m\start)(m\end), inner sep=3pt, label=below:{$a^1_{\the\numexpr\i+1\relax},\dots, a^{n-2}_{\the\numexpr\i+1\relax}$}] (rectM\i) {};
		}
		
		\foreach \i in {1,3,5} {
			\pgfmathtruncatemacro{\start}{3*\i}
			\pgfmathtruncatemacro{\end}{3*\i+2}
			\node[draw=blue, rounded corners, fit=(m\start)(m\end), inner sep=3pt, label=above:{$a^1_{\the\numexpr\i+1\relax},\dots, a^{n-2}_{\the\numexpr\i+1\relax}$}] (rectM\i) {};
		}

		\foreach \i in {0,2} {
			\node[circle, draw, fill=white, inner sep=1.5pt] (c\i a) at (3+\i*4,0) {};
			\node[circle, draw, fill=black, inner sep=1.5pt] (c\i b) at (3+\i*4+2,0) {};
			\node (c\i c) at (3+\i*4+1,0) {$\dots$};
			
			\draw (c\i a) -- (c\i c);
			\draw (c\i b) -- (c\i c);

			\node[draw=blue, rounded corners, fit=(c\i a)(c\i b), inner sep=3pt, label=below:{$b^1_{\the\numexpr\i+1\relax},\dots,b^{n-3}_{\the\numexpr\i+1\relax}$}] (rectB\i) {};	}
		
\foreach \i in {1,3} {
	\node[circle, draw, fill=black, inner sep=1.5pt] (c\i a) at (3+\i*4,0) {};
	\node[circle, draw, fill=white, inner sep=1.5pt] (c\i b) at (3+\i*4+2,0) {};
	\node (c\i c) at (3+\i*4+1,0) {$\dots$};
	
	\draw (c\i a) -- (c\i c);
	\draw (c\i b) -- (c\i c);

	\node[draw=blue, rounded corners, fit=(c\i a)(c\i b), inner sep=3pt, label=below:{$b^1_{\the\numexpr\i+1\relax},\dots,b^{n-3}_{\the\numexpr\i+1\relax}$}] (rectB\i) {};	}

		\draw (roota) -- (m0.north);
		\draw (rootb) -- (m20.north);
		
		\draw (m0.south) -- (c0a.north);
		\draw (m3.south) -- (c1a.north);
		\draw (m6.south) -- (c2a.north);
		\draw (m9.south) -- (c3a.north);
		
		\draw (m11.south) -- (c0b.north);
		\draw (m14.south) -- (c1b.north);
		\draw (m17.south) -- (c2b.north);
		\draw (m20.south) -- (c3b.north);
		
		\end{tikzpicture}
	\caption{\label{fig:figure3} The configuration $A$ for $n\geq 5$ odd.}	 
\end{figure}
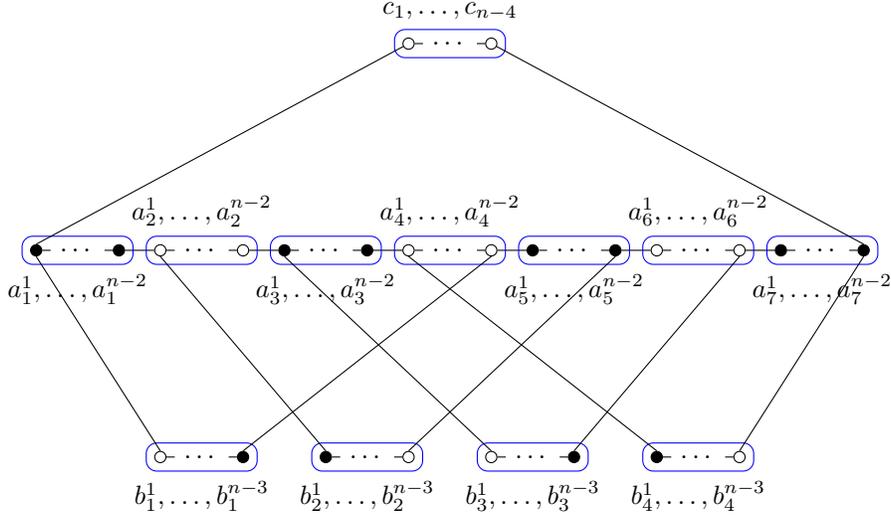
	
	\noindent We first verify that $A$ is confined partial $n$-gon.
	 
	\medskip
	
	\begin{claim}
		$A$ is a partial $n$-gon.
	\end{claim}
	\begin{proof}
		It suffices to verify that $A$ does not contain cycles of length $<2n$. Let $\gamma$ be a cycle in $A$, then we have the following possible configurations.
		
		\smallskip
		\noindent (i) $\gamma$ is a cycle containing $c_1,\dots,c_{n-4}$ and no subchain of the form $b^1_k,\dots,b^{n-3}_k$. In this case $\gamma$ is the cycle containing all elements $a_k^i$ for all $1\leq k\leq 7$ and $1\leq i \leq n-2$. It follows that $\gamma$ has length $(n-4)+7(n-2)=8n-18$. Since for all $n\geq 5$ we have $8n-18\geq 2n$, $\gamma$ is an admissible cycle.
		
		\smallskip
		\noindent (ii) $\gamma$ is a cycle containing $c_1,\dots,c_{n-4}$ and one subchain of the form $b^1_k,\dots,b^{n-3}_k$. By construction of $A$ it follows that $\gamma$ has length $(n-4)+3(n-2)+(n-3)+2=5n-11$. Since for all $n\geq 5$ we have $5n-11\geq 2n$ it follows that $\gamma$ is an admissible cycle.
		
		\smallskip
		\noindent (iii) $\gamma$ is a cycle containing $c_1,\dots,c_{n-4}$ and two subchains of the form $b^1_k,\dots,b^{n-3}_k$. It is straightforward to verify that the shortest cycle of this kind is the cycle
		\[c_1\,\vert\, a^1_1\,\vert\, b_1^1\,\vert\, \dots \,\vert\, b_1^{n-3}\,\vert\, a_4^{n-2}\,\vert\, \dots \,\vert\, a_4^1\,\vert\, b^1_4 \,\vert\, \dots \,\vert\, b_4^{n-3}\,\vert\, a^{n-2}_7 \,\vert\, c_{n-4}\,\vert\, \dots \,\vert\, c_1   \]
		which has length $(n-4)+2(n-3)+(n-2)+2=4n-10$. Since $4n-10\geq 2n$ whenever $n\geq 5$, $\gamma$ is an admissible cycle.
		
		\smallskip
		\noindent (iii) $\gamma$ is a cycle not containing $c_1,\dots,c_{n-4}$ but containing two subchains of the form $b^1_k,\dots,b^{n-3}_k$. Then $\gamma$ contains (at least) two blocks of $n-2$ points and two blocks of $n-3$ points, e.g., $\gamma$ is the chain 	
		\[a^1_2\,\vert\, b_2^1\,\vert\, \dots \,\vert\, b_2^{n-3}\,\vert\, a_5^{n-2}\,\vert\, a_6^1\,\vert\, \dots \,\vert\, a_6^{n-2}\,\vert\, b_3^{n-3}\,\vert\, \dots \,\vert\, b_3^1\,\vert\, a^1_3\,\vert\, a^{n-2}_2\,\vert\, \dots \,\vert\, a^1_2.   \]
		Then $\gamma$ has length $2(n-2)+2(n-3)+2=4n-8$ and we have  for $n\geq 5$ that $4n-10\geq 2n$ always holds. So $\gamma$ is an admissible cycle.
		
		\smallskip
		\noindent (iv) $\gamma$ is a cycle not containing $c_1,\dots,c_{n-4}$ and containing only one subchain of the form $b^1_k,\dots,b^{n-3}_k$. Then it is immediate to verify that $\gamma$ has length $4(n-2)+(n-3)=5n-11$. Since $5n-11\geq 2n$ for all $n\geq 5$ we have that $\gamma$ is admissible.
		
		\smallskip
		\noindent Since all cycles in $A$ are of one of the forms (i)-(iv) above, it follows that $A$ does not contain cycles of length $<2n$, and thus that it is a partial $n$-gon.
	\end{proof} 
	
	\begin{claim}
		$A$ is confined.
	\end{claim} 
\begin{proof}
	 Firstly, notice that by construction every element in $A$ has valency at least 2. Moreover, every chain of $n-2$ elements contains at least one element with valency $\geq 3$. It follows directly from Definition~\ref{ngons:def_hyperfree} that $A$ does not contain any hyperfree tuple (recall also Remark~\ref{remark:confined},). We conclude that $A$ is confined.
\end{proof}

We finally define a configuration $B$ satisfying Lemma~\ref{the_crucial_lemma}. Let $A_1$ and $A_2$ be two disjoint copies of $A$ and let $c^1_1,\dots, c^1_{n-4}$, $c^2_1,\dots, c^2_{n-4}$ be the two copies of $c_1,\dots,c_4$ from $A_1$ and $A_2$, respectively. We let $B=A_1\cup A_2\cup \{d_1,\dots,d_{n-4}\}$ where $d_1,\dots,d_{n-4}$ is a chain with $d_1\,\vert \dots\,\vert d_{n-4}$ and such that $d_1$ and $d_{n-4}$ are two lines incident respectively to $c_{n-4}^1$ and $c_{n-4}^2$.

	\medskip
	\begin{claim}
		$B$ is a finite, confined, non-degenerate partial $n$-gon containing $A_1 \cup A_2$ and such that the only copies of $A$ in $B$ are $A_1$ and  $A_2$.
	\end{claim} 
	\begin{proof}
		 Firstly, since every path from one element of $A_1$ to one element to $A_2$ must go through $d_1$, it follows that every cycle in $B$ must already be contained in $A_1$ or $A_2$. Thus $B$ does not contain cycles of length $<2n$ and is also a partial $n$-gon. Since the elements in $d_1,\dots d_{n-4}$ all have two incidences and they also do not belong to any clean arc (given that the elements $c_{n-4}^1$ and $c_{n-4}^2$ have valency 3) we also have that $B$ is confined. Moreover, since the two copies of the element $a^{n-2}_1$ from $A_1,A_2$ have distance $\geq n+2$ for all $n\geq 5$, it follows from Remark~\ref{non-degeneracy} that $B$ is non-degenerate.  Finally, to see that $B$ does not contain any other copy of $A$ different from $A_1$ and $A_2$, we notice the following: in $A$ every two elements are connected by two disjoint paths, while in $B$ every path from some $x\in A_1$ and $y\in A_2$ must necessarily pass through $d_1$. It follows that an embedding $\pi:A\to B$ cannot contain elements both in $A_1$ and $A_2$. Also, $\pi(A)$ cannot contain any element from $\{d_1,\dots, d_{n-4}\}$, since every element in $A$ has valency $\geq 2$, but every $D\subseteq \{ d_1,\dots, d_{n-4} \}$ contains at least one element of valency 1 in $A_i\cup D$ for $i\in\{1,2\}$. Thus $\pi(A)=A_1$ or $\pi(A)=A_2$, proving our claim and concluding our construction.
	\end{proof}
\end{construction}

\begin{construction}[$n=2m$ for $m\geq 3$]\label{even_case} It remains to consider the even case for $n\geq 6$. We provide a construction very similar to the one from Construction~\ref{odd_case}, but we make some key adjustments due to the different parity of the clean arcs in this setting. Fix $n=2m$ for some $m\geq 3$, we define a partial confined $n$-gon $A$ as follows. For each $1\leq k\leq 5 $ we let $a_k^1\,\vert\,\dots \,\vert\, a_k^{n-2}$ be a chain such that $a_k^1$ is a point and $a_k^{n-2}$ is a line (this is possible because $n$ is even). Additionally, we let $a_k^{n-2}\,\vert\, a_{k+1}^{1}$ for all $1\leq k\leq 4$. For $1\leq k\leq 3$ we let $b_k^1\,\vert\,\dots\,\vert\,b_k^{n-3}$ be a chain with both $b_k^1$ and $b_k^{n-3}$ lines. Additionally, we let $a_1^1\,\vert\, b_1^1$, $a_2^1\,\vert\, b_2^1$, $a_3^1\,\vert\, b_3^1$, and $a_3^{n-3}\,\vert\, b_1^{n-3}$, $a_4^{n-3}\,\vert\, b_2^{n-3}$, $a_5^{n-3}\,\vert\, b_3^{n-3}$. Finally, we let $c_1\,\vert\,\dots\,\vert\, c_{n-4}$ be a chain with $c_1$ a line and $c_{n-4}$ a point such that $a_1^1\,\vert\, c_1$, $a_5^{n-2}\,\vert\, c_{n-4}$. We draw this configuration below.
	
	\begin{figure}[h]
		\centering
		\begin{tikzpicture}[scale=0.55]
		
		% --- Topmost root node (3 dots with ellipsis) ---
		\node[circle, draw, fill=white, inner sep=1.5pt] (roota) at (8,10) {};
		\node (rootc) at (9.5,10) {$\dots$};
		\node[circle, draw, fill=black, inner sep=1.5pt] (rootb) at (11,10) {};
		\draw (roota) -- (rootc);
		\draw (rootb) -- (rootc);
		
		% Rectangle for root
		\node[draw=blue, rounded corners, fit=(roota)(rootb), inner sep=3pt, label=above:{$c_1,\dots,c_{n-4}$}] (rectRoot) {};
		
		% --- Middle Layer (21 dots with \dots) ---
		\foreach \i in {0,2,4,6,8,10,12,14,16,18} {
			\node[circle, draw, fill=black, inner sep=1.5pt] (m\i) at (\i*1,5) {};
		}
	\foreach \i in {3,7,11,15,19} {
		\node[circle, draw, fill=white, inner sep=1.5pt] (m\i) at (\i*1,5) {};
	}
		\foreach \i in {1,5,9,13,17} {
			\node (m\i) at (\i*1,5) {$ \dots $};
		}
		
		% Connect middle layer chain
		\foreach \i in {0,...,18} {
			\draw (m\i) -- (m\the\numexpr\i+1\relax);
		}
		
		% Rectangles around groups of 3 in the middle
		\foreach \i in {0,2,4} {
			\pgfmathtruncatemacro{\start}{4*\i}
			\pgfmathtruncatemacro{\end}{4*\i+3}
			\node[draw=blue, rounded corners, fit=(m\start)(m\end), inner sep=3pt, label=above:{$a^1_{\the\numexpr\i+1\relax},\dots, a^{n-2}_{\the\numexpr\i+1\relax}$}] (rectM\i) {};
		}
		
		\foreach \i in {1,3} {
			\pgfmathtruncatemacro{\start}{4*\i}
			\pgfmathtruncatemacro{\end}{4*\i+3}
			\node[draw=blue, rounded corners, fit=(m\start)(m\end), inner sep=3pt, label=below:{$a^1_{\the\numexpr\i+1\relax},\dots, a^{n-2}_{\the\numexpr\i+1\relax}$}] (rectM\i) {};
		}
		
		% --- Bottom Layer (4 groups with \dots) ---
		\foreach \i in {0,...,2} {
			\node[circle, draw, fill=white, inner sep=1.5pt] (c\i a) at (3+\i*5,0) {};
			\node[circle, draw, fill=white, inner sep=1.5pt] (c\i b) at (3+\i*5+3,0) {};
			\node (c\i c) at (3+\i*5+1.5,0) {$\dots$};
			\draw (c\i a) -- (c\i c);
			\draw (c\i b) -- (c\i c);
			
			% rectangle around each bottom group
			\node[draw=blue, rounded corners, fit=(c\i a)(c\i b), inner sep=3pt, label=below:{$b^1_{\the\numexpr\i+1\relax},\dots,b^{n-3}_{\the\numexpr\i+1\relax}$}] (rectB\i) {};
		}
		
		% --- Edges ---
		\draw (roota) -- (m0.north);
		\draw (rootb) -- (m19.north);
		
		\draw (m0.south) -- (c0a.north);
		\draw (m4.south) -- (c1a.north);
		\draw (m8.south) -- (c2a.north);
		
		\draw (m10.south) -- (c0b.north);
		\draw (m14.south) -- (c1b.north);
		\draw (m18.south) -- (c2b.north);		
		\end{tikzpicture}
		\caption{\label{fig:figure4} The configuration $A$ for $n\geq 6$ even.}
	\end{figure}
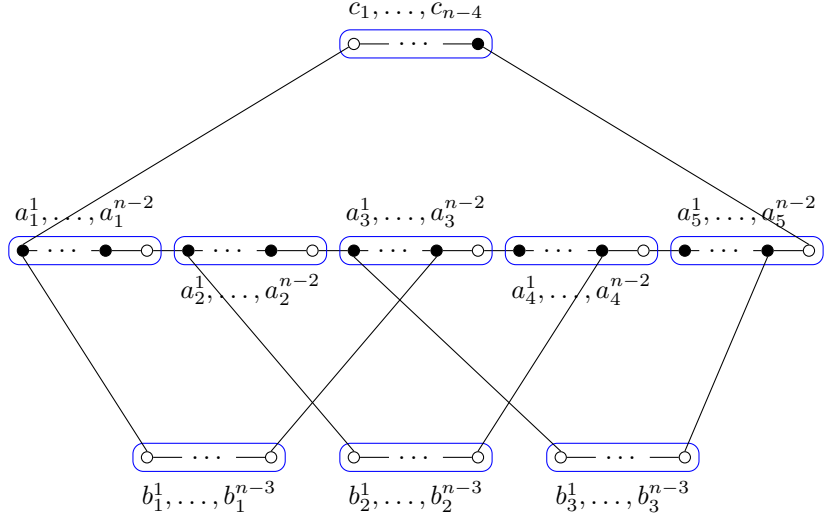

	\noindent We verify that $A$ is confined partial $n$-gon.
	
	\medskip
	
	\begin{claim}
		$A$ is a partial $n$-gon.
	\end{claim}
	\begin{proof}
		It suffices to verify that $A$ does not contain cycles of length $<2n$. Let $\gamma$ be a cycle in $A$, then we have the following possible configurations.
		
		\smallskip
		\noindent (i) $\gamma$ is a cycle containing $c_1,\dots,c_{n-4}$ and no subchain of the form $b^1_k,\dots,b^{n-3}_k$. In this case $\gamma$ must be the cycle containing all elements $a_k^i$ for all $1\leq k\leq 5$ and $1\leq i \leq n-2$. It follows that $\gamma$ has length $(n-4)+5(n-2)=6n-14$. Since for all $n\geq 5$ we have $6n-14\geq 2n$, $\gamma$ is an admissible cycle.
		
		\smallskip
		\noindent (ii) $\gamma$ is a cycle containing $c_1,\dots,c_{n-4}$ and only one subchain of the form $b^1_k,\dots,b^{n-3}_k$. By construction one can verify that $\gamma$ has length $(n-4)+(n-3)+2(n-2)+3=4n-8$ and for all $n\geq 5$ we have $4n-8\geq 2n$, thus $\gamma$ is an admissible cycle.
		
		\smallskip
		\noindent (iii) $\gamma$ is a cycle containing $c_1,\dots,c_{n-4}$ and two subchains of the form $b^1_k,\dots,b^{n-3}_k$. It is then straightforward to verify that the shortest cycle of this kind is the cycle
		\[c_1\,\vert\, a^1_1\,\vert\, b_1^1\,\vert\, \dots \,\vert\, b_1^{n-3}\,\vert\, a_3^{n-3}\,\vert\, \dots \,\vert\, a_3^1\,\vert\, b^1_3 \,\vert\, \dots \,\vert\,  b_3^{n-3}\,\vert\, a^{n-3}_5 \,\vert\, a^{n-2}_5 \,\vert\, c_{n-4}\,\vert\, \dots \,\vert\, c_1   \]
		which has length $(n-4)+3(n-3)+3=4n-10$. Since $4n-10\geq 2n$ whenever $n\geq 5$, $\gamma$ is an admissible cycle.
		
		\smallskip
		\noindent (iii) $\gamma$ is a cycle not containing $c_1,\dots,c_{n-4}$ but containing two subchains of the form $b^1_k,\dots,b^{n-3}_k$. Then $\gamma$ contains (at least) two blocks of $n-2$ points and two blocks of $n-3$ points, e.g., $\gamma$ is a cycle of the form	
		\[a^1_2\,\vert\, b_2^1\,\vert\, \dots \,\vert\, b_2^{n-3}\,\vert\, a_4^{n-3}\,\vert\, a_4^{n-2}\,\vert\, a_5^1 \,\vert\, \dots \,\vert\, a_5^{n-3}\,\vert\, b_3^{n-3}\,\vert\, \dots \,\vert\, b_3^1\,\vert\, a^1_3\,\vert\, \dots \,\vert\, a^1_2.   \]
		Then $\gamma$ has length greater than $2(n-2)+2(n-3)=4n-10$ and we have that $4n-10\geq 2n$ exactly when $n\geq 5$. So $\gamma$ is an admissible cycle.
		
		\smallskip
		\noindent (iv) $\gamma$ is a cycle not containing $c_1,\dots,c_{n-4}$ and containing one subchains of the form $b^1_k,\dots,b^{n-3}_k$. It is immediate to check that $\gamma$ has length at least $2(n-2)+2(n-3)=4n-10$. Since $4n-10\geq 2n$ when $n\geq 5$ we have that $\gamma$ is admissible.
		
		\smallskip
		\noindent Since all cycles in $A$ are of one of the forms (i)-(iv) above, it follows that $A$ does not contain cycles of length $<2n$, and thus it is a partial $n$-gon.
	\end{proof}
	
	\begin{claim}
		$A$ is confined.
	\end{claim}
	\begin{proof}
		By construction, every element in $A$ has valency at least 2, and every chain of $n-2$ elements in $A$ contains at least one element with valency $\geq 3$. As in Remark~\ref{remark:confined} we have that $A$ does not contain any hyperfree tuple and we thus conclude that it is confined.
	\end{proof}
	
	Finally, we define a configuration $B$ satisfying Lemma~\ref{the_crucial_lemma} as follows. Let $A_1$ and $A_2$ be two disjoint copies of $A$ and consider the subchains $c^1_1,\dots, c^1_{n-4}$ and $c^2_1,\dots, c^2_{n-4}$ which are the copies of $c_1,\dots, c_{n-4}$ from $A_1$ and $A_2$, respectively. We let $B=A_1\cup A_2\cup \{d_1,\dots,d_{n-4}\}$ where $d_1,\dots,d_{n-4}$ is a chain with $d_1\,\vert \dots\,\vert d_{n-4}$ and such that $d_1$ a point incident to $c^1_1$ and $d_{n-4}$ a line incident to $c_{n-4}^2$. By reasoning as in Construction~\ref{odd_case} it can be shown that $B$ satisfies Lemma~\ref{the_crucial_lemma}. This completes the proof of Lemma~\ref{the_crucial_lemma}.
\end{construction}

\subsection{The Borel reduction}\label{Sec:borel_reduction_ngon}

We use the configuration introduced in the previous section to define a Borel reduction from the Borel space of countable trees (known to be Borel complete) to the Borel space of countable generalised $n$-gons, for all $3\leq n <\omega$. By classical results from descriptive set theory, also the class of countable trees without endpoints is known to be Borel complete. Since trees without endpoints are essentially \emph{firm} $\infty$-gons (namely, $\infty$-gons with each element having valency $\geq 2$), the end result of this section is that \emph{all} families of generalised $n$-gons for $n\in[3,\infty]$ are Borel complete. On the other hand, it is easy to see that the theory of infinite bipartite graphs is $\aleph_0$-categorical, whence countable 2-gons are obviously not Borel complete, and they have in fact the simplest possible isomorphism relation.  We recall the Borel completeness of countable trees, originally proved in \cite{friedman}, and we then define a Borel reduction from countable trees to countable generalised $n$-gons, for all $3\leq n <\omega$. Recalling Definition~\ref{def:graphs}, we always consider trees as connected and acyclic bipartite graphs.

\begin{fact}[{\cite[Thm. 1.1.1]{friedman}},  {\cite[Cor.~13.2.2; Ex.~13.2.2]{gao}}]\label{completeness:trees} Countable trees are Borel complete. Moreover, also countable trees with no endpoints are Borel complete.
\end{fact}

\begin{construction}\label{ngon_construction} Let $3 \leq n < \omega$. Given a tree $\Gamma = (T, E)$, we define a partial $n$-gon $P^-_{(\Gamma, n)}$ as follows:
	\begin{enumerate}[(i)]
	\item we add disjoint copies $A_v$ of $A$ from \ref{the_crucial_lemma}, for every $v \in T$;
	\item for every edge $e\in E$ between $v_1,v_2\in T$ we add a copy $B(e)$ of $B$ from \ref{the_crucial_lemma} where $A_{v_1} \cup A_{v_2}$ play the role of $A_1 \cup A_2$.
	\end{enumerate}
Then we define $P_{(\Gamma, n)}$ as $F(P^-_{(\Gamma, n)})$, i.e., as the free completion of the partial $n$-gon $P^-_{(\Gamma, n)}$ (recall Definition~\ref{def_free_completion}). 

\smallskip \noindent To show that this construction is sound, all that we need to show is that $P^-_{(\Gamma, n)}$ is a partial $n$-gon, as then by Definition~\ref{def_free_completion} it follows that $P_{(\Gamma, n)}$ is a generalised $n$-gon. Now, since $\Gamma$ is a tree, it does not contain any cycle, and thus by construction every cycle in $P^-_{(\Gamma, n)}$ must already be contained in some configuration of the form $B(e)$ for some edge $e\in E$. Since by Lemma~\ref{the_crucial_lemma} we have that $B(e)$ is a partial $n$-gon, it follows that every cycle $\gamma$ in $P^-_{(\Gamma, n)}$ has length at least $2n$ and so that  $P^-_{(\Gamma, n)}$ is also a partial $n$-gon. Therefore, it follows immediately from Definition~\ref{def_free_completion} and Remark~\ref{non-degeneracy} that $P_{(\Gamma, n)}$ is a generalised $n$-gon.   Additionally, if $\Gamma$ has size $\lambda \geq \aleph_0$, then it follows immediately by our construction that $P_{(\Gamma, n)}$ also has size $\lambda$. As a matter of facts, we also have (if $\Gamma$ contains at least one edge, namely  if $|\Gamma|\geq 2$) that since the structure $B$ from Lemma~\ref{the_crucial_lemma} is always non-degenerate, in the free completion of $P^-_{(\Gamma, n)}$ we recursively add new neighbours to every element, and thus we obtain that every element in $P_{(\Gamma, n)}$ has infinite valency (of size $\lambda=|\Gamma|+\aleph_0$).
\end{construction}

\begin{lemma}\label{key_lemma_As}
	Let $3 \leq n < \omega$. For every tree $\Gamma=(T,E)$ we have that every copy of $A$ (from Constructions~\ref{case_n=3}--\ref{even_case}) in $P_{(\Gamma, n)}$ is of the form $A_v$ for some $v\in \Gamma$.
\end{lemma}
\begin{proof}
	By Lemma~\ref{the_crucial_lemma} it suffices to show that any embedding $\pi(A)$ of $A$ into $P_{(\Gamma, n)}$ is already contained in a configuration of the form $B(e)$. We prove this claim by distinguishing the case $n=3$ and all the remaining cases $n\geq 4$. For any $v\in \Gamma$ we write $A_v$ for the corresponding copy of $A$ in $P_{(\Gamma, n)}$ (as in \ref{ngon_construction}), and we write $p^e,\ell^e,d^e,\dots$ for the  elements adjoined in $B(e)$ following Constructions~\ref{case_n=3}--\ref{even_case}.  The reader is advised to keep in mind the configurations and the drawings from  Section~\ref{ngon:finitary_construction}.
	
	\smallskip
	\noindent For the case $n=3$ the claim follows by reasoning similarly as in the proof from Construction~\ref{case_n=3} that the only copies of $A$ in $B$ are $A_1$ and $A_2$. Suppose in fact that $\pi(A)$ is a copy of $A$ in $P^-_{(\Gamma, 3)}$ which is not fully contained in some configuration $B(e)$. Then, since every two elements in $A$ are joined by a path, it follows immediately that $\pi(A)$ must contain the point $p^{e_1}$ or the line $\ell^{e_1}$ for some edge $e_1\in E$. We assume without loss of generality that $p^{e_1}\in \pi(A)$. Then, since this element must have valency 3 in $\pi(A)$, it follows that $\ell^{e_1},a^v_1,a^u_1$ must also be in $\pi(A)$, where $v,u\in T$ are the vertices of the edge $e_1$. Since $a^v_1\in \pi(A)$ and $a_1\in A$ has valency 3, it follows that $\pi(A)$ contains exactly two other neighbours of $a^v_1$. If $\pi(A)$ contains $a^v_2,a^v_{10}$ or $a^v_{14}$ then we can reason exactly as in Construction~\ref{case_n=3} and obtain a contradiction. Suppose thus there are two more edges $e_{2}$ and $e_{3}$ in $\Gamma$ with $v$ as one of their vertices, and that $p^{e_2},p^{e_3}\in \pi(A)$, and by construction $p^{e_2}\,\vert\,a^v_1$,  $p^{e_3}\,\vert\, a^v_1$. By the same argument as before, $p^{e_2},p^{e_3}$ must also have valency 3 in $\pi(A)$, which forces in particular $\ell^{e_2},\ell^{e_3}\in \pi(A)$, where by construction $\ell^{e_2}\,\vert\,a^v_4$,  $\ell^{e_3}\,\vert\, a^v_4$. Since also $\ell^{e_2}$ and $\ell^{e_3}$ must have valency 3 in $\pi(A)$, and since $\Gamma$ is a tree, it follows that $a^v_4\in \pi(A)$. But then we obtain that in $\pi(A)$ the elements $a^v_1$ and $a^v_4$ belong to three different paths of length 3, i.e., the path $a^v_1\;\vert p^{e_1} \;\vert   \ell^{e_1} \;\vert  \; a^v_4$, the path $a^v_1\;\vert p^{e_2} \;\vert   \ell^{e_2} \;\vert  \; a^v_4$ and also the path $a^v_1\;\vert p^{e_3} \;\vert   \ell^{e_3} \;\vert  \; a^v_4$. Differently, in $A$ there are only two paths of length 3 between any two elements. Hence $\pi(A)$ cannot contain elements from different copies of $B$ in $P^-_{(\Gamma, 3)}$, and we thus conclude from Lemma~\ref{the_crucial_lemma} that every copy of $A$ in $P_{(\Gamma, 3)}$ is of the form $A_v$ for some $v\in \Gamma$.
	
\smallskip 
\noindent  For all remaining cases with $n\geq 4$ it suffices to notice the following. In Construction~\ref{case_n=4}, Construction~\ref{odd_case} and Construction~\ref{even_case}, the configuration $B$ is obtained so that every path between some elements from $A_1$ and some element from $A_2$ must contain all elements in $B\setminus (A_1\cup A_2)$. Therefore, since $\Gamma$ is a tree, it follows from Construction~\ref{ngon:finitary_construction} that for any two distinct edges $e_1,e_2$ from $\Gamma$, any two elements $x\in B(e_1)$ and $y\in B(e_2)$ belong to some common cycle in $P^-_{(\Gamma, n)}$ if and only if $B(e_1)\cap B(e_2)=A_v$ and $x,y\in A_v$ for some vertex $v\in \Gamma$. Now, since every two elements from $A$ belong to a common cycle in $A$ (for all $n\geq 4$), it follows that every copy of $A$ in $P^-_{(\Gamma, n)}$ must be already contained in some configuration $B(e)$ for some edge $e$ from $\Gamma$. We conclude that for any $n\geq 3$ every copy of $A$ must be already contained in some $B(e)$ and so, by Lemma~\ref{the_crucial_lemma}, that every copy of $A$ in $P_{(\Gamma, n)}$ is of the form $A_v$ for some $v\in \Gamma$.
\end{proof}

We next show that the function mapping a tree $\Gamma$ to the generalised $n$-gon $ P_{(\Gamma, n)} $ preserves isomorphism types and their negations. Additionally, we also prove that $\Gamma$ is first-order interpretable in $P_{(\Gamma, n)}$, although this will not be needed in the proof of Borel completeness.

\begin{proposition}\label{lemma_interpre} Let $3 \leq n < \omega$. For every two  trees $\Gamma_1$ and $\Gamma_2$ we have that $\Gamma_1 \cong \Gamma_2$ if and only if $P_{(\Gamma_1, n)} \cong P_{(\Gamma_2, n)}$. Additionally, for every tree $\Gamma$ we have that $\Gamma$ is first-order interpretable in $P_{(\Gamma, n)}$.
\end{proposition}
	\begin{proof}
	If  $\Gamma_1,\Gamma_2$ are two trees such that $\Gamma_1\cong \Gamma_2$, then it follows immediately by Construction~\ref{ngon_construction} that  $P^-_{(\Gamma_1, n)} \cong P^-_{(\Gamma_2, n)}$, and thus also that $P_{(\Gamma_1, n)} \cong P_{(\Gamma_2, n)}$. Conversely, suppose that $f:P_{(\Gamma_1, n)} \cong P_{(\Gamma_2, n)}$ is an isomorphism. Since $P^-_{(\Gamma_1, n)}$ is confined, it follows that $f(P^-_{(\Gamma_1, n)})$ is also confined and, by Lemma~\ref{lemma_confined}, we must have that $f(P^-_{(\Gamma_1, n)})\subseteq P^-_{(\Gamma_2, n)}$. By the same argument applied to $f^{-1}$, we similarly obtain that $f^{-1}(P^-_{(\Gamma_2, n)})\subseteq P^-_{(\Gamma_1, n)}$, and thus we conclude that $f:P^-_{(\Gamma_1, n)}\cong P^-_{(\Gamma_2, n)}$. Now, by Lemma~\ref{key_lemma_As}, every copy of $A$ in the partial $n$-gon $P^-_{(\Gamma, n)}$  must be of the form $A_v$ for some $v\in \Gamma$. It follows that every copy $A_{v_1}$ of $A$ in $P^-_{(\Gamma_1, n)}$ with $v_1\in \Gamma_1$ must be sent by $f$ to a copy $A_{v_2}$ of $A$ in $P^-_{(\Gamma_2, n)}$, for some $v_2\in \Gamma_2$. It follows that two configurations $A_{v_1}$ and $A_{w_1}$ belong to a configuration $B(e)$ in $P^-_{(\Gamma_1, n)}$, with $v_1,w_1$ vertices in $\Gamma_1$ and $e$ an edge in $\Gamma_1$, whenever the corresponding configurations $f(A_{v_1})=A_{v_2}$, $f(A_{w_2})=A_{w_2}$ in $P^-_{(\Gamma_2, n)}$ belong to an analogous configuration $B(e')$ for some edge $e'$ from $\Gamma_2$. It follows that $f$ induces a graph isomorphism $\hat{f}$ between the underlying trees $\Gamma_1$ and $\Gamma_2$, witnessing that $\Gamma_1\cong \Gamma_2$.

	\smallskip \noindent We next observe that $\Gamma$ is first-order interpretable in $P_{(\Gamma,n)}$. Since the constructions from Section~\ref{ngon:finitary_construction} are finitary, there is a first-order formula $\phi_n$ in finitely-many variables that specifies the atomic diagram of $A$ from Lemma~\ref{the_crucial_lemma}. We use this formula to define a first-order interpretation $\Psi$ from a tree $\Gamma$ into $P_{(\Gamma, n)}$. Denote by $\phi_n(P_{(\Gamma, n)})$ the subset of tuples in $P_{(\Gamma, n)}$ satisfying the formula $\phi_n$; this means essentially that a tuple $\bar{a}\in P_{(\Gamma, n)}$ has the isomorphism type of $A$ (and in particular, by Lemma~\ref{the_crucial_lemma}, that it is a confined configuration). Therefore, it follows from Lemma~\ref{lemma_confined}, that if $\bar{a} \in \phi_n(P_{(\Gamma, n)})$, then $\bar{a}\in P^-_{(\Gamma, n)}$. Then, since $\bar{a}$ has the isomorphism type of $A$, it  follows again from Lemma~\ref{key_lemma_As} that $\bar{a}$ is exactly of the form $A_v$ for some $v\in \Gamma$. We can thus define the interpretation $\Psi: \phi_n(P_{(\Gamma, n)})\to \Gamma$ by letting $\Psi(\bar{a})=v$, where $v$ is the unique vertex of $\Gamma$ such that $\bar{a}$ is an enumeration of $A_v$. By construction of $P_{(\Gamma, n)}$ this map is clearly surjective. Now, since $B$ is finite, it can also be described by a first-order formula $\psi_n$, and we thus have that there is an edge $e$ between $ \Psi(\bar{a})$ and $\Psi(\bar{b})$ if and only if $P^-_{(\Gamma, n)}\models \psi_n(\bar{a},\bar{b})$. It follows that $\Psi$ is a first-order interpretation of $\Gamma$ in $P_{(\Gamma, n)}$.
	\end{proof}

	\begin{theorem} Let $3 \leq n < \omega$, then the Borel space of generalised $n$-gons with domain~$\omega$ is Borel complete.
\end{theorem}
	\begin{proof} Notice that Construction~\ref{ngon_construction} sends a tree $\Gamma$ of size $\aleph_0$ to a generalised $n$-gon $P_{(\Gamma, n)}$ of size $\aleph_0$. Moreover, by  Proposition~\ref{lemma_interpre}, it preserves isomorphisms and their negations. Finally, we see (exactly as in \cite[Thm.~2]{pao}) that to know a finite substructure of $ P_{(\Gamma,n)} $ it suffices to  know a finite part of $\Gamma$, and thus the map $\Gamma\to P_{(\Gamma,n)}$ is Borel. We conclude from \ref{completeness:trees} that for any $3 \leq n < \omega$ the Borel space of generalised $n$-gons with domain~$\omega$ is Borel complete.
	\end{proof}

\section{Buildings}\label{sec_building}

In this last section we use the ``free construction'' introduced by Ronan in \cite{ronan}  for buildings without rank $3$ spherical residues to extend the Borel completeness of generalised $n$-gons from the previous section to arbitrary buildings without rank~$3$ spherical residues. We refer the reader to \ref{def:basic_buildings} for the definition of buildings without rank 3 spherical residues, and we next recall some additional definitions.

\begin{notation}\label{notation_for_buildings}
	If $\Sigma(W,S)$ is a Coxeter complex, then by Definition~\ref{coxeter_complex} it is immediate to see that its chambers are of the form $\{w\}$ for some element $w\in W$ of the associated Coxeter group, as these are of the form $w \langle J \rangle_W$ for $\emptyset = J \subseteq S$. We thus often identify chambers of $\Sigma(W,S)$ with elements in $W$ and simply write $W$ for the complex $\Sigma(W,S)$. If $(W,S)$ is a Coxeter system and $w\in W$ is a chamber in $\Sigma(W,S)$, we denote by $\ell_S(w) \coloneqq \ell(w)$ the minimal length of a word in the alphabet $S$ representing $w$ in $W$ (so the word length with respect to $(W, S)$). For every $n<\omega$ we define $W_n\coloneqq\{w\in W : \ell(w)\leq n\}$. Finally,  if $\Delta$ is a chamber complex and $x\in\Delta$ is a simplex, then we denote by $\mrm{st}_\Delta(x)$ the simplicial complex consisting of all simplices $y\in \Delta$ such that $y\geq x$ (cf. \cite[\S 1.1]{tits_book}).
\end{notation}

The following definitions are from \cite[p.~243]{ronan}.

\begin{definition}
	Let $\Gamma$ be a generalised $n$-gon, viewed as a building (recall Example~\ref{example:ngons_as_buildings}), and let $x\in \Delta$ be a chamber. Let $k\in \mathbb{Z}$, then the chamber complex $\Gamma_x$ consisting of all chambers $y\in\Delta$ with distance $\leq n-k$ from $x$ (and their faces) is called a \emph{$k$-denuded $n$-gon with centre $x$}. For $k\geq 1$ the complex $\Gamma_x$ is a partial $n$-gon with no cycles (where we consider $\Gamma_x$ as a graph with edges between chambers induced by the adjacency relation, cf.~\ref{notation_chamber complexes}). For $k\leq 0$ we have that $\Gamma_x=\Gamma$.
\end{definition}

We recall the following fact about the existence of \emph{projections} in a buildings. The following fact was originally established by Tits in \cite[3.19]{tits_book} (cf. also \cite[Prop.~ 4.95]{abra}).

\begin{fact} Suppose $\Delta$ is a building, $x\in \Delta$ is a simplex and $c\in \Delta$ is a chamber. Then there exists a unique chamber, denoted by $\mrm{proj}_c(x)$, which contains $x$ and such that $d(\mrm{proj}_c(x),c/\Delta)<d(y,c/\Delta)$ for any other chamber $y\in \Delta$ containing $x$.
\end{fact}

 The following remark is from \cite[p.~245]{ronan}.

\begin{remark}\label{remark:denuded_in_coxeter}
	Let $(W,S)$ be a Coxeter system and let $x\in \Delta \coloneqq \Sigma(W,S)$ be a simplex of corank 2, then $\mrm{st}_\Delta(x)$ is a generalised $m_{ij}$-gon, where $\{i,j\}$ is the cotype of $x$. Similarly, for $n<\omega$, we have that $\mrm{st}_\Delta(x)\cap W_n$ is a $k$-denuded $m_{ij}$-gon, where
	 $$k=\ell(\mrm{proj}_1(x))+m_{ij}-n.$$
\end{remark}

Our key goal is to define, for any Coxeter diagram $D=(I,E)$ with no rank~3 spherical residues, a Borel map from the class of countable trees to the class of countable buildings with Coxeter diagram $D$. Starting from a tree $\Gamma$, we use Ronan's construction from \cite{ronan} to define a ``free'' building $\Delta_\Gamma$ of type $D$ whose isomorphism type is determined by $\Gamma$. Recall that the \emph{parameters} of a generalised $n$-gon $A$, viewed as a building, is the pair $(q_i,q_j)$ which specify the numbers of chambers to which every panel of cotype $\{i\}$ and cotype $\{j\}$ belong, respectively (cf.~\cite[p.~242]{ronan}). A key feature of Ronan's construction is that it allows us to construct $\Delta_\Gamma$ so that its rank 2 residues of type $\{i,j\}$ are isomorphic to some pre-determined $m_{ij}$-gon, as long as they have the same parameters $q_i$ and $q_j$. We apply this construction using the generalised $n$-gons $P_{(\Gamma, n)}$ from \ref{ngon_construction}. In the following we simply describe how Ronan's construction works applied to our specific context, and we refer the reader to \cite{ronan} for a detailed proof of the fact that this construction actually yields a building.

\begin{construction}[Ronan \cite{ronan}]\label{construction:ronan}
	Let $D=(I,E)$ be a Coxeter diagram with no rank 3  spherical residues and let $\mathcal{F}=\{A_{ij} :  i\neq j\in I\}$ be a family of generalised $n$-gons satisfying the following conditions:
	\begin{enumerate}[(a)]
		\item if $m_{ij} = 2$, then $A_{ij}$ is the complete bipartite graph of size $\aleph_0$;
		\item if $m_{ij} = \infty$, then $A_{ij}$ is the free pseudoplane (i.e., the unique countable tree where every vertex has infinite valency);
		\item if $m_{ij} \geq 3$, then $A_{ij}$ is a generalised $m_{ij}$-gon where every element has infinite valency.
	\end{enumerate}
	We define the associated ``free'' building $\Delta_\mathcal{F}$ of type $D$ using Ronan's construction from \cite{ronan}, relatively to the family $\mathcal{F}$.  Notice that since in every generalised $n$-gon from $\mathcal{F}$ every element has infinite valency, it follows that every panel in each $A_{ij}$ is a face of infinitely many chambers. This means that these generalised $n$-gons have the same \emph{parameters}, ensuring that they satisfy the compatibility conditions required for Ronan's construction (cf. \cite[p.~243]{ronan}). We describe the construction of the free building $\Delta_\mathcal{F}$ and refer to \cite{ronan} for further details and the proof that $\Delta_\mathcal{F}$ is indeed a building. It is quite important to keep in mind that, in the following construction, we always view (partial) generalised $n$-gons as simplicial complexes, following Example~\ref{example:ngons_as_buildings}.

	\medskip \noindent  Let $W$ be the Coxeter complex with diagram $ D=(I,E) $ (recalling Definition~\ref{coxeter_complex} and Notation~\ref{notation_for_buildings}). We define the building $\Delta_\mathcal{F}$ with Coxeter diagram $D$ as the union of an increasing chain $(\Delta_n)_{n<\omega}$ of simplicial complexes. Notice also that we view each simplicial complex $\Delta_n$ as a chamber complex with an underlying type function $\mrm{tp}:\Delta_n\to I$. It follows that the key notions from \ref{def:basic_buildings} apply also in this context, even if $\Delta_n$ is not (yet) a building. Recall also the notion of adjacency from \ref{notation_chamber complexes} and the fact that we identify elements of $W$ with chambers in $\Sigma(W, S)$. For each $n<\omega$, we assume inductively that there is an adjacency-preserving map $\rho_n:\Delta_n\to W$ (i.e., a map which preserves $i$-adjacency in both directions for all $i\in I$, cf. \cite[p.~51]{tits_book}) such that $\rho_n\subseteq \rho_m$ whenever $n<m$, and satisfying the following condition:
	\smallskip
	
	\begin{enumerate}[$(C_n)$]
		\item 
		\begin{enumerate}
		\item the image of $\rho_n$ is $W_n$;
		\item if $y$ is a simplex of $\Delta_n$ of cotype $\{i,j\}$, $\rho_n \restriction \mrm{st}_{\Delta_n}(y)$ is the canonical retraction with centre $z$, where $z$ is the unique element of $\mrm{st}_{\Delta_n}(y)$ such that $\rho_n(z) = \mrm{proj}_1(\rho_n(y))$ (cf. \cite[2.4]{tits_book} and \cite[p. 244]{ronan}).
		\item furthermore, for $y$ and $z$ as in clause (b), we have that $\mrm{st}_{\Delta_n}(y)$ is a $k$-denuded $m_{ij}$-gon, for $k=\ell(\mrm{proj}_1(\rho_n(y)))+m_{ij}-n$.
	\end{enumerate}
	\end{enumerate}
	
	\smallskip \noindent  We proceed as follows. Let $\Delta_0=\{c \}$ consist of a single chamber $c$ and suppose inductively that $\Delta_m$ is defined for all $m\leq n$, we define $\Delta_{n+1}$ and $\rho_{n+1}:\Delta_{n+1}\to W$ by considering all chambers $w\in W_{n+1}\setminus W_{n}$. Notice that since $w\in W_{n+1}\setminus W_{n}$ is a word on the alphabet $I$ of length $n+1$, $W_n$ contains a word $v$ such that $w=vi$ for some $i\in I$. Then the standard coset $w\langle i\rangle_{W}$ contains both $w$ and $v=wi^{-1}=wi$, thus $w$ and $v$ are $i$-adjacent (recall \ref{coxeter_complex}). It follows that every $w\in W_{n+1}\setminus W_{n}$ is adjacent to at least one element from $W_n$. Additionally, since $D$ does not contain any spherical residue of rank 3, it follows that $w$ is adjacent to \emph{at most two} chambers in $W_n$ (see \cite[Lemma 2]{ronan}). It thus suffices to consider the following cases.
	
	\medskip \noindent \underline{Case 1}. $w\in W_{n+1}\setminus W_{n}$ is adjacent to exactly one chamber $v\in W_n$. 
	\newline Let $i\in I$ be the unique index for which $w$ is $i$-adjacent to $v$.  For every chamber $u\in \Delta_n$ such that $u\in \rho_n^{-1}(v)$ we add in $\Delta_{n+1}$ a new set of chambers $\{x_t : 1\leq t <\omega \}$ such that they are all $i$-adjacent to $u$ and to each other. Additionally, we let them be $j$-adjacent only to themselves for every $i\neq j\in I$. We extend $\rho_n$ to $\rho_{n+1}$ by letting  $\rho_{n+1}(x_t)=w$ for all $1\leq t <\omega$. It is then clear that $\rho_{n+1}$ is still an adjacency-preserving map between $\Delta_{n+1}$ and $W_{n+1}$.
	
	\medskip \noindent \underline{Case 2}. $w\in W_{n+1}\setminus W_{n}$ is adjacent to exactly two chambers $v_1,v_2\in W_n$. 
	\newline Let $i,j\in I$ be the two indices in $I$ such that $w$ is $i$-adjacent to $v_1$ and $w$ is $j$-adjacent to $v_2$ for some $i,j\in I$. Notice also that we must have that $i\neq j$, for otherwise it follows that $v_1=v_2$. Let $x\in \Sigma(W,S)$ be the (unique) simplex of cotype $\{i,j\}$ which is a face of $w$ (namely, $x=w\land v_1\land v_2$) and consider any $y\in \rho_{n}^{-1}(x)$. Notice that, for any such $y$, $\mrm{st}_{\Delta_n}(y)$ is a $\{i,j\}$-residue of $\Delta_{n+1}$. Moreover, it follows from Condition~$(C_n)$ that the simplex $\mrm{st}_{\Delta_n}(y)$ is a $1$-denuded $m_{ij}$-gon with centre $z$ (for $z$ as in clause (b) of condition $(C_n)$), i.e., it is the simplex containing all chambers with distance $\leq m_{ij}-1$ from $z$ (cf.~\cite[p.~247]{ronan}). It is thus possible to complete each $\mrm{st}_{\Delta_n}(y)$  into a $m_{ij}$-gon isomorphic to $A_{ij}$: we adjoin chambers $\{z_t : t<\omega\}$ such that $y\leq  z_t$ for all $t<\omega$ and so that we obtain a simplicial complex $C(y)\coloneqq \mrm{st}_{\Delta_n}(y)\cup \{z_t : t<\omega\}$ of type $\{i,j\}$ which forms an $m_{ij}$-gon isomorphic to $A_{ij}$. For each $k\notin \{i,j\}$ we simply let each chamber $z_t$ be $k$-incident only to itself. Finally, we extend $\rho_n$ to $\rho_{n+1}$ by letting $\rho_{n+1}(z_t)=w$ for all $t<\omega$.
	
	\medskip \noindent We let $\Delta_{n+1}$ be the simplicial complex obtained by performing the construction from Case 1 or Case 2 to every $w\in W_{n+1}\setminus W_{n}$, and we let $\rho_{n+1}$ also be defined by the two former steps. It was proved by Ronan in \cite{ronan} that $\rho_{n+1}$ and $\Delta_{n+1}$ together satisfy property $(C_{n+1})$. Let then $\Delta_\mathcal{F}=\bigcup_{n<\omega}\Delta_n$ and $\rho=\bigcup_{n<\omega}\rho_n$, it follows from \cite[Lemma 1]{ronan} that $\rho$ is an adjacency-preserving map from $\Delta_\mathcal{F}$ onto $W$, and thus that $\Delta_\mathcal{F}$ is a building with Coxeter diagram $D=(I,E)$. Additionally, it follows immediately by construction that for every distinct $i,j\in I$, any $\{i, j\}$-residue of $B$ is isomorphic to $P(i, j)$ (cf. \cite[p.~243]{ronan}). The isomorphism type of the building $\Delta_\mathcal{F}$ thus depends exclusively on the isomorphism types of the generalised $m_{ij}$-gons $P(i, j)$ for $i,j \in I$ with $i\neq j$, which was \mbox{specified in conditions (a)-(c) of the statement.}
\end{construction}

We conclude the article by using together Construction~\ref{ngon_construction} and Construction~\ref{construction:ronan} to prove our main theorem.

\begin{theorem}\label{main_th}
	For every Coxeter diagram $D=(I,E)$ with no rank $3$ residues of spherical type and such that $D$ has at least one edge labelled by $n\in [3,\infty]$, the space of buildings of type $D$ with domain $\omega$ is Borel complete.
\end{theorem}
\begin{proof} 
	Suppose first that  $D$ has at least one edge labelled by $3\leq n <\omega$.	Given a tree $\Gamma$ of size $\aleph_0$, we consider  for every $i\neq j\in I$ with $m_{ij}$ the associated generalised $n$-gons $P_{(\Gamma,m_{ij})}$ given by Construction~\ref{ngon_construction}. Then, let $\mathcal{F}_\Gamma$ be the family made of all the generalized $n$-gons $P_{(\Gamma,m_{ij})}$, for $m_{ij}\geq 3$, and let $P_{(\Gamma,m_{ij})}$ be the complete bipartite graph of size $\aleph_0$ or free pseudoplane for $m_{ij}=2$ or $m_{ij}=\infty$, respectively. Thus $\mathcal{F}_\Gamma$ satisfies the conditions from Construction~\ref{construction:ronan}, and we can therefore perform Ronan's construction with $\mathcal{F}_\Gamma$ as input. We let $\Delta_\Gamma$ be the building obtained in this way.  We view $\Delta_\Gamma$ as a first-order structures in the multi-sorted language with the nonempty subsets $J\subseteq I$ as sorts and the order relation $\leq$. The elements of $\Delta_\Gamma$ are its simplices, $x\in \Delta_\Gamma$ belongs to the sort $J$ if its type as a simplex is $J$, and for $x,y\in \Delta_\Gamma$ we have $x\leq y$ if $x$ is a face of $y$. Notice that since $|\Gamma|=\aleph_0$ we have that also $\Delta_\Gamma$ has size $\aleph_0$.
	
	\smallskip
	\noindent We prove that the map $\Gamma \mapsto \Delta_\Gamma$ is a Borel reduction from the space of countable trees to the space of buildings of type $D$. Firstly, we show  that $\Gamma \mapsto \Delta_\Gamma$ preserves isomorphisms and their negation. The first claim is clear from Construction~\ref{construction:ronan}. Concerning the latter, suppose that $\Delta_{\Gamma_1} \cong \Delta_{\Gamma_2}$ and let $b$ be a simplex of cotype $\{i, j\}$ in $\Delta_{\Gamma_1}$. Now, $\mrm{st}_{\Delta_{\Gamma_1}}(b) $ is a $\{i, j\}$-residue in $\Delta_{\Gamma_1}$ and thus it is isomorphic to $P_{(\Gamma_1, m_{ij})}$. Similarly, the complex $\mrm{st}_{\Delta_{\Gamma_2}}(f(b)) $ is a $\{i, j\}$-residue in $\Delta_{\Gamma_2}$ and thus it is isomorphic to $P_{(\Gamma_2, m_{ij})}$. Therefore, since $f$ is an isomorphism, we have that: $$P_{(\Gamma_1, m_{ij})} \cong P_{(\Gamma_2, m_{ij})},$$
	but by Proposition~\ref{lemma_interpre} this happens if and only if $\Gamma_1 \cong \Gamma_2$, proving our claim. 
	
	\smallskip \noindent Thus it remains to show that the map $\Gamma \mapsto \Delta_\Gamma$ is Borel. To this end, we can view this map as the composition of two functions: the first one assigning $\Gamma$ to the structure $M_{\Gamma}$  consisting of the disjoint union of all the $P_{(\Gamma, m_{ij})}$, together with labels indicating for each $P_{(\Gamma, m_{ij})}$ the associated indices $i$ and $j$; the second map assigning $M_{\Gamma}$ to $\Delta_\Gamma$ as in Construction~\ref{construction:ronan}. Clearly, the first map is Borel, so it suffices to show that $M_{\Gamma}\to \Delta_{\Gamma}$ is Borel as well. To this end, notice that to know whether two chambers $x,y\in \Delta_{\Gamma}$ are $i$-adjacent, it suffices to know at which step of Construction~\ref{construction:ronan} they were introduced and, in case they were introduced during the Case 2 step from Construction~\ref{construction:ronan}, also to which element of the corresponding generalised $n$-gon $P_{(\Gamma, m_{ij})}$ they correspond to. Thus, if follows that also the map $M_{\Gamma}\mapsto \Delta_\Gamma$ is Borel, showing that the function mapping $\Gamma$ to $\Delta_{\Gamma}$ is Borel as well. 
	
	\smallskip
	\noindent It remains to consider the case when $D$ has no edges labelled by $3\leq n <\omega$. In this case we define a Borel reduction from the class of countable trees with no endpoints, which is also Borel complete by Fact~\ref{completeness:trees}. By the previous assumptions it follows that $D$ has an edge $\{i, j\}$ labelled by $\infty$.  We then consider the building $\Delta^\infty_\Gamma$ obtained as in Construction~\ref{construction:ronan} but by choosing for $m_{ij}=2$ the unique countable bipartite graph, and for $m_{ij}=\infty$ the very same tree $\Gamma$, which by Remark~\ref{remark:spherical_buildings} is a building of rank 2. Notice in particular that in $\Gamma$ every maximal flag is adjacent to infinitely-many others, thus $\Gamma$ has parameters compatible with the complete bipartite graph of size $\aleph_0$, and thus they can be used together in Ronan's construction (cf.~\ref{construction:ronan}). Reasoning exactly as before, it follows that the map $\Gamma\mapsto \Delta^\infty_\Gamma$ is Borel, and that for two buildings $\Delta^\infty_{\Gamma_1} $ and $\Delta^\infty_{\Gamma_2} $ we have that $\Delta^\infty_{\Gamma_1} \cong \Delta^\infty_{\Gamma_2}$ if and only if $\Gamma_1\cong \Gamma_2$. This completes our proof.
\end{proof}

\end{document}